\documentclass[12pt]{amsart}
\usepackage[utf8]{inputenc}
\usepackage{setspace}
\usepackage[english]{babel}
\usepackage{mathabx}
\usepackage{framed}
\usepackage{soul}
\usepackage[all]{xy}
\usepackage[margin=1in]{geometry}
\usepackage{arydshln}
\usepackage[colorinlistoftodos]{todonotes}
\usepackage{pb-diagram}
\usepackage[mathscr]{euscript}
\usepackage{multicol}
\usepackage{stmaryrd} 
\usepackage{empheq} 
\usepackage{tikz-cd}
\tikzset{
  symbol/.style={
    draw=none,
    every to/.append style={
      edge node={node [sloped, allow upside down, auto=false]{$#1$}}}
  }
}

\usepackage{amsmath,amsfonts,amssymb,amsthm,mathrsfs,latexsym,mathtools}
\usepackage{commath}
\usepackage[T1]{fontenc}
\usepackage{hyperref}
\usepackage{caption} 
\usepackage{subcaption} 

\DeclareMathAlphabet{\mathbbmsl}{U}{bbm}{m}{sl}

\usetikzlibrary{matrix}

\newcommand{\C}{\mathbb{C}} 
\newcommand{\Z}{\mathbb{Z}}
\newcommand{\R}{\mathbb{R}}

\newcommand{\bbQ}{\mathbb{Q}}

\newcommand{\disc}{\mathrm{disc}}

\newcommand{\prim}{\textup{prim}}

\newcommand{\Pic}{\textup{Pic}}

\newcommand {\OO} {{\mathcal O}}
\newcommand{\PP} {\mathbb{P}}
\newcommand{\cS}{\mathcal{S}}
\renewcommand{\S}{\cS}
\newcommand{\E}{\mathcal{E}}
\newcommand{\Kum}{\textup{Kum}}

\newcommand{\cF}{\mathcal{F}}
 
\newcommand{\cC}{\mathcal{C}}
\newcommand{\cA}{\mathcal{A}} 

\newcommand{\cX}{\mathcal{X}}

\newcommand{\cH}{\mathcal{H}}

\newcommand {\cY} {\mathcal{Y}}

\newcommand {\Q} {{\mathbb Q}}

\newcommand{\sat}{\mathrm{sat}}

\usepackage{amscd} 

\usepackage{enumitem} 
\usepackage{comment}

\newtheorem{theorem}{Theorem}[section]
\newtheorem{definition}[theorem]{Definition}
\newtheorem{proposition}[theorem]{Proposition}
\newtheorem{corollary}[theorem]{Corollary}
\newtheorem{question}[theorem]{Question}
\newtheorem{example}[theorem]{Example}
\newtheorem{lemma}[theorem]{Lemma}

\newtheorem{remark}[theorem]{Remark}

\newtheorem{notation}[theorem]{Notation}

\newtheorem{thm}{Theorem}

\tikzset{commutative diagrams/.cd,
mysymbol/.style={start anchor=center,end anchor=center,draw=none}
}

\title{Failure of the Invariant Cycle Theorem over $\Z$}

\author{Donu Arapura}
\address{Department of Mathematics, Purdue University, 150 N. University Street, West Lafayette, IN 47907, U.S.A.}
\email{arapura@purdue.edu}
\thanks{The first author was partially supported by a grant from the Simons Foundation, and the second author by the NSF grant DMS-2302548}

\author{Fran\c{c}ois Greer}
\address{Department of Mathematics, Michigan State University,
619 Red Cedar Road, East Lansing, MI 48824, U.S.A.}
\email{greerfra@msu.edu}

\author{Yilong Zhang}
\address{Department of Mathematics, University of Georgia,
110 Carlton St, Athens, GA 30602, U.S.A.}
\email{yilong.zhang@uga.edu}

\date{May 13, 2026}
\begin{document}
\maketitle
\begin{abstract}
We initiate a study of the local invariant cycle theorem with integral coefficients for 1-parameter semistable families of varieties. We show that it always holds for $H^1$, and it holds for $H^2$ if the general fiber has trivial Albanese variety. The latter generalizes results of Friedman, Griffiths, and Scattone on K3 surfaces and I-surfaces.

We construct the first example of a semistable family which fails the local (and global) invariant cycle theorems with integral coefficients. The family has constant period map associated to $H^2$, and its smooth fibers are algebraic surfaces with $p_g=q=1$; in particular, they have non-trivial Albanese varieties. 
The surfaces in the family have maximal Picard rank and minimal discriminant, and they are closely related to Vinberg's most algebraic K3 surface. 
Our construction also generalizes the Shioda--Inose construction for rational double covers of K3 surfaces.
\end{abstract}

\setcounter{tocdepth}{1} 
\tableofcontents

\section{Introduction}

The classical local invariant cycle theorem, stated below, was conjectured by Griffiths \cite{Griffiths} in 1970 and proved by Clemens and Schmid shortly thereafter \cite{Cle77}\cite[pp 121-122]{GriffithsSchmid}.
Professor Griffiths kindly informed us that many of these ideas were discussed in a seminar at Berkeley in the 1960's. A different proof was found by Deligne \cite[Th\'eor\`eme 3.6]{Deligne-WeilII}. Both proofs are non-trivial; the first relies on the existence of limit mixed Hodge structures and the second on the Weil conjectures.

\begin{theorem}[LICT$_\bbQ$]\label{ICTQ}
    Let $f:\mathcal{X}\to \Delta$ be a flat family of complex projective varieties over a  disk. Suppose $\mathcal{X}$ is smooth and $f$ is smooth over the punctured disk $\Delta^*$. Let $X_t$ be a general smooth fiber, then for each $n$, the restriction map 
$$H^n(\mathcal{X},\bbQ)\twoheadrightarrow H^n(X_t,\bbQ)^{\pi_1(\Delta^*)}$$
    surjects onto the invariant part.
\end{theorem}

For semistable families, i.e., where the central fiber $X_0$ is a reduced normal crossing divisor and the total space is smooth, the theorem follows from the exactness of the Clemens--Schmid sequence \cite{Cle77}:
\begin{equation}\label{intro_eqn_CS-sequence}
    H^n(X_0,\bbQ)\to H^n(X_t,\bbQ)\xrightarrow{N} H^n(X_t,\bbQ),
\end{equation}
where $N=\log(T)$ is the logarithm of the monodromy operator $T$, and $\ker (N)=\ker(T-I)$ is the invariant part $H^n(X_t,\bbQ)^{\pi_1(\Delta^*)}$. More generally, a much stronger form of  Theorem \ref{ICTQ} follows from the decomposition theorem \cite[Theorem 6.2.5]{BBD82}.

However, all the previous arguments only work rationally. A natural question to ask is
\begin{question}[\textup{LICT}$_\Z$]\label{question_ICTZ}
  Let $f:\mathcal{X}\to \Delta$ be a semistable family.  Does the local invariant theorem still hold with integral coefficients? In other words, is the restriction map  
    $$r:H^n(\mathcal{X},\Z)\to H^n(X_t,\Z)^{\pi_1(\Delta^*)}$$ surjective?
\end{question}
Note that the semistability assumption is necessary. Without it, an easy counterexample can be constructed as follows. Let $\mathcal{X}\to \Delta$ be a local elliptic surface with a multiple fiber of multiplicity $m>1$ over $0$. Any class in the image of $H^2(\mathcal X,\Z)\to H^2(X_t,\Z)$ must be divisible by $m$. In particular, the generator of $H^2(X_t,\Z)$, which is invariant, cannot lift to $H^2(\mathcal{X},\Z)$. One can preclude this sort of counterexample by working with semistable families. 

\subsection{Positive results}
Question \ref{question_ICTZ} is equivalent to asking whether the Clemens--Schmid sequence \eqref{intro_eqn_CS-sequence} is exact over $\Z$.
This was first considered by Friedman and Scattone while studying degenerations of K3 surfaces \cite{Fri84,FriSca86}. They showed that \textup{LICT}$_\Z$ holds for semistable families of K3 surfaces, and more recently Friedman and Griffiths obtained a similar result for I-surfaces \cite{FriGri24}.

We generalize their results to optimal form in low cohomological degrees:
\begin{theorem}[\textup{LICT}$_\Z$ for $H^1$, Theorem \ref{ssh1}]\label{intro_thm_H^1} Let $f:\mathcal{X}\to \Delta$ be a semistable family of complex projective varieties over a disk.  Then the \textup{LICT}$_\Z$ holds for $H^1$.

\end{theorem}

\begin{theorem}[\textup{LICT}$_\Z$ for $H^2$, Theorem \ref{thm_H^2}]\label{intro_thm_H^2}
     Let $f:\mathcal{X}\to \Delta$ be a semistable family of complex projective varieties over a disk. If the general fiber $X_t$ has trivial Albanese, or equivalently $q(X_t)=0$, then \textup{LICT}$_\Z$ holds for $H^2$.
\end{theorem}
In particular, Theorem \ref{intro_thm_H^2} applies to all projective hypersurfaces, Calabi-Yau varieties, hyperk\"ahler varieties, as well as their finite quotients. Note that it also applies to torsion classes in $H^2$, which appear for instance in Enriques surfaces and Godeaux surfaces.

To prove the result for $H^1$, we first prove LICT$_{\Z}$ for a semistable family of abelian varieties, using the Mumford construction and toric moment maps. This leads to LICT$_{\Z}$ for any semistable family of curves by passing to the relative compactified Jacobian. Lastly, we use a relative hyperplane section theorem to reduce to the case of curves (cf. Theorem \ref{thm_reduction-hyp}). To prove the result for $H^2$, we use the fact that the Clemens--Schmid sequence is a slicing of the Wang sequence and the exact sequence of the pair $(\mathcal X,\mathcal X\smallsetminus X_0)$, both of which are exact over $\Z$.



\subsection{Negative results}


Next, we answer Question \ref{question_ICTZ} in the negative for surfaces with $q=1$, showing that Theorem \ref{intro_thm_H^2} is sharp.

\begin{theorem}\label{main_thm}
The local invariant cycle ``theorem'' ($\textup{LICT}_\Z$) fails over $\Z$ for a semistable family. More precisely,
    there exists a semistable degeneration $\mathcal{X} \to \Delta$ of surfaces with $p_g=q=1$ and torsion-free cohomology, such that the restriction map 
    $$H^2(\mathcal{X},\Z)\to H^2(X_t, \Z)^{\pi_1(\Delta^*)}$$
    is not surjective.
\end{theorem}
The smooth fibers of the family are elliptic surfaces with $p_g=q=1$. We refer to such surfaces as elliptic-elliptic surfaces, following \cite{GZ}. They carry K3-type Hodge structures on $H^2$ and non-trivial Hodge structures on $H^1$. We construct a family of such surfaces over a complete curve with finite monodromy on $H^2$, and infinite monodromy on $H^1$. The family has several degenerations. Restricting to a neighborhood of a specific singular fiber, and then passing to a semistable model yields a family $\mathcal{X}\to \Delta$ with constant VHS on $H^2$, such that the image
of $H^2(\mathcal{X},\Z)$ is a proper subgroup of $ H^2(X_t, \Z)$.
To prove this, we show that the specialization map induces an injection on transcendental parts
$$T_{X_0}\hookrightarrow  T_{X_t},$$
 which preserves the cup product pairings. 
We find that the central fiber $X_0$ has a component birational to a specific elliptic  K3
surface with discriminant $48$.
However, we show that the discriminant
of $T_{X_t}$ is $3$.
Hence, the image of the restriction map on transcendental lattices has index $4$, so the LICT$_{\Z}$ fails. On the other hand, if one restricts to Picard lattices, LICT$_{\Z}$ holds for divisor classes in the family (cf. Corollary \ref{cor_ICTZ-divisor}).

\begin{corollary}
  The Clemens--Schmid sequence fails to be exact over $\Z$ for the semistable family in Theorem \ref{main_thm}. 
\end{corollary}
In particular, the decomposition theorem fails over $\Z$. This should be compared with the failure of the
decomposition theorem with torsion coefficients studied by Williamson \cite{Williamson17}.

\subsection{Global invariant cycle theorem}

For any proper surjective family $\mathcal{X}\to B$ over a projective base $B$, the global invariant cycle theorem holds over $\bbQ$ \cite{Deligne-HodgeII}, stated formally as Theorem \ref{intro_thm_GICTQ}. 
One can ask whether the same global statement holds integrally (GICT$_\Z$). It is classically known to hold for Lefschetz pencils on a smooth projective variety \cite{Lam81}. 

For semistable families over a complete curve, we show that GICT$_\Z$ holds for $H^1$, and for $H^2$ provided that $q(X_t)=0$ (see Corollary \ref{cor_GICT_H^1}). However, it is still false for $H^2$ in general: 





\begin{theorem}\label{intro_thm-failureGICTZ}
    The global invariant cycle ``theorem'' (\textup{GICT}$_{\Z}$) fails over $\Z$ for a semistable family. More precisely, there is a smooth projective threefold $\mathcal X$, a smooth projective curve $C$, and a proper semistable surjective morphism $\mathcal X\to C$, such that the restriction map
    $$H^2(\mathcal X,\Z)\to H^2(X_t,\Z)^{\pi_1(U,t)}$$
    is not surjective, where $U\subset C$ is the open subset over which the family is smooth.
\end{theorem}

Note that in our example, the Leray spectral sequence still degenerates over $\Z$ for the smooth fibration $\mathcal X_U\to U$ because $U$ is a non-compact Riemann surface homotopy equivalent to a 1-dimensional CW complex, so the invariant cycles lift to $H^2(\mathcal X_U,\Z)$. However, it fails to lift to $H^2(\mathcal X,\Z)$ due to 
 Theorem \ref{main_thm}. 

\subsection{The construction}\label{subsect:construction}

The family $\mathcal X \to \Delta$ used in the proof of Theorem \ref{main_thm} has constant period map.
Each smooth fiber is Hodge-theoretically associated with a K3 surface $Y$, which is one of Vinberg's ``most algebraic'' K3 surfaces \cite{Vin83} with discriminant $d=3$, in the sense that there is a Hodge isometry $T_{X_t}\cong T_Y$. They are related by the following diagram, formally analogous to a Shioda--Inose structure \cite{ShiIno77,Morrison84}:

\begin{figure}[ht]
    \centering
\begin{equation*}
\begin{tikzcd}
X_t\arrow[dr,dashed,"\pi_1"] && Y\arrow[dl,dashed,"\pi_2"']\\
&Z
\end{tikzcd}
\end{equation*}
\end{figure}

  Shioda and Inose \cite{ShiIno77} showed that $Y$ admits a rational map of degree two onto the Kummer surface $Z=\Kum(E_\omega\times E_\omega)$ associated to the self-product of the elliptic curve $E_\omega = \C/(\Z+\Z \omega)$ with $j=0$.
The rational map $\pi_2$ arises from a quadratic base change of an elliptic fibration on $Z$ branched at two special fibers ($IV^*,I_0^*$).

We show that the quadratic base change branched at {\it three} special fibers ($II^*,IV^*,I_0^*$), plus an additional smooth fiber lying over $t$ produces an elliptic-elliptic surface $X_t$. As $t$ varies, the Hodge structure on $H^2(X_t,\Z)$ is locally constant.
The three singular members of the family occur when $t$ collides with each of the special values lying under singular fibers of $Z$. Each of these singular surfaces contains a reduced component birational to a K3 surface; one of these K3 surfaces is $Y$. This forces the discriminant of $T_{X_t}$ to be 3. 
On the other hand, when $t$ collides with the special value lying under the $I_0^*$ fiber, the limiting K3 is different; it has discriminant $48$. This allows us to deduce Theorem \ref{main_thm}. A second counterexample can be constructed similarly, where $Y$ is replaced with Vinberg's second most algebraic K3, which has discriminant $4$ (cf. Theorem \ref{thm_2nd-counterexample}).



\subsection{Additional questions} It would be interesting to seek other counterexamples to the semistable LICT$_\Z$ using torsion classes in $H^k(X_t,\Z)$; the families studied here have torsion-free $H^2(X_t,\Z)$. We also wonder whether LICT$_\Z$ holds in {\it all} cohomological degrees for semistable degenerations of abelian varieties.

\subsection*{Outline} In Section \ref{sec_prelim}, we recall the preliminary notions of transcendental Hodge structures, elliptic surfaces, and Shioda--Inose structures. 
Next, we prove some positive results on the invariant cycle theorem over $\Z$, Theorems \ref{intro_thm_H^1} and \ref{intro_thm_H^2}. In Section \ref{sec_H^1}, we discuss the LICT$_{\Z}$ for $H^1$ and prove Theorem \ref{ssh1}. In Section \ref{sec_postiveresults}, we discuss the LICT$_{\Z}$ for $H^2$ and prove Theorem \ref{intro_thm_H^2}. We also prove the global version in  Corollary \ref{cor_GICT_H^1}. 

The rest of the paper concerns negative results, namely Theorems \ref{main_thm} and \ref{intro_thm-failureGICTZ}. In Section \ref{sec_constantHodge}, we study quadratic base change of elliptic surfaces and construct a smooth family of elliptic-elliptic surfaces with constant VHS on $H^2$. In Section \ref{sec_degenerations}, we extend the smooth family to a proper family with controlled singular fibers. In Section \ref{sec_LICTZ}, we prove Theorem \ref{main_thm}. In Section \ref{sec_GICTZ}, we prove Theorem \ref{intro_thm-failureGICTZ}. In Section \ref{sec_addcounterexample}, we construct a second counterexample using the same techniques, but a different Shioda--Inose input. 

In Appendix \ref{sec_SNC}, we discuss some topological results on SNC varieties. 
In Appendix \ref{sec_periodmap}, we find correspondences between the elliptic-elliptic surfaces in question and their associated K3 surfaces. In Appendix \ref{sec_divisor}, we discuss the invariant cycle theorem for divisor classes.

\subsection*{Acknowledgments} We would like to thank Valery Alexeev, Philip Engel, Phillip Griffiths, and Matt Kerr for useful communications.

\section{Preliminaries}\label{sec_prelim}
\subsection{Transcendental Hodge structures}\label{subsec_transcendental}
In this paper, the word \emph{surface}, without any further qualification, will mean
 a smooth projective surface  over $\C$. If $S$ is a surface,  the N\'eron--Severi group $NS(S)$ is the subgroup of divisor classes in $H^2(S,\Z)$. The intersection pairing on $NS(S)/\textup{torsion}$ is nondegenerate, so it carries a lattice structure. Its orthogonal complement $T_S$ in $H^2(S,\Z)$ is called the transcendental lattice. 
 $$T_S:= NS(S)^\perp\subset H^2(S,\Z).$$
By Wu's formula, $T_S$ is always an even lattice, meaning that for every $x\in T_S$, $x\cdot x\in 2\Z$. One has $|\det(NS(S))|=|\det(T_S)|$, and this positive integer is called the \textit{discriminant} of $S$ and is denoted by $\disc(S)$.
 
 The complexification $T_S\otimes \C$ contains $H^{2,0}(S,\C)$, and it carries a Hodge filtration induced from $H^2(S,\C)$. Hence, $T_S$ (resp. $T_{S,\bbQ}$) carries a weight two $\Z$-Hodge structure (resp. $\bbQ$-Hodge structure) induced from $H^2(S,\Z)$.  We call $\phi:T_S\to T_{S'}$ a Hodge isometry if $\phi$ is an isometry as a lattice and preserves the Hodge filtration.

\begin{lemma}\label{lemma_trans-irreducible}
The transcendental Hodge structure $T_{S,\bbQ}\subset H^2(S,\bbQ)$ is the smallest $\bbQ$  sub-Hodge structure whose complexification contains $H^{2,0}(S)$.  Furthermore $$T_S = T_{S,\bbQ}\cap H^2(S,\Z)/\textup{torsion}.$$
\end{lemma}
\begin{proof}
By the Lefschetz $(1,1)$ theorem, $NS(S)_\bbQ$ is the maximal Hodge structure contained in $H^{1,1}(S)$. hence the first statement follows from the fact that $H^{2,0}(S)$ is orthogonal to $H^{1,1}(S)$. The second statement follows from the fact that  $T_S$ is primitive in $H^2(S,\Z)/\text{torsion}$.
\end{proof}


Given a dominant rational map $\pi:S\dashrightarrow S'$ between surfaces. There is a pullback $\pi^*$ and a pushforward $\pi_*$ (of free abelian groups)
$$\pi^*:T_{S'}\to T_S,\ \pi_*:T_S\to T_{S'}.$$
To see this, one can resolve the indeterminacy of $\pi$:
$$S\xleftarrow{\sigma}\tilde{S}\xrightarrow{p} S',$$
where $\sigma$ is birational. We can define $\pi^*:=\sigma_*\circ p^*$ and $\pi_*:=p_*\sigma^*$ on transcendental lattices. It is well-defined since the transcendental lattice is birationally invariant (both $\sigma_*$ and $\sigma^*$ induce isomorphisms $T_S\cong T_{\tilde{S}}$), and the definition is independent of the resolution $\sigma$. We remark that $\pi^*$ and $\pi_*$ do not preserve the intersection pairings in general.

 When $p_g(S)=p_g(S')$, then both $\pi^*$ and $\pi_*$ are injective and the image is a finite index sublattice. Thus there is an isomorphism of (unpolarized) $\bbQ$-Hodge structures: $$T_S\otimes\bbQ\cong T_{S'}\otimes \bbQ.$$

\subsection{Elliptic surfaces}
An elliptic surface is a smooth projective surface $S$ together with a map $f:S\to C$ to a smooth projective curve $C$, whose general fiber is an elliptic curve. We call an elliptic surface relatively minimal if there is no ($-1$)-curve contained in a fiber of $f$. If $f$ has a section, then the set of sections of $f$ form an abelian group $MW(S/C)$, called the Mordell--Weil group, which is finitely generated as long as $f$ has a singular fiber. The rank of this group is called the Mordell--Weil rank. 
There is a  $\bbQ$-lattice structure on the torsion-free part of the Mordell--Weil group, called the Mordell--Weil lattice. One refers to \cite{SchShi08, SchShi19} for details.

Without further qualification, we will assume that all elliptic surfaces are relatively minimal and admit a section. The possible singular fibers in an elliptic surface were classified by Kodaira, N\'eron, and Tate \cite[Section 5.4]{SchShi19} into the following types:
$$II, II^*, III,III^*,IV, IV^*,I_n,I_n^* \ \textup{with}\ n\ge0 .$$
These singular fibers occur as the Kodaira--N\'eron models for any one-parameter family of elliptic curves over $\Delta^*$. Among all the singular fiber types, we are particularly interested in the following subset featuring the superscript $^*$:

\begin{definition}\normalfont
    The singular fiber types $II^*,III^*,IV^*,I_n^*$, $n\ge 0$ are called \textit{star fibers}.
\end{definition}

They are all reducible and non-reduced; the dual graph encoding the incidence of all components not meeting the zero section
 is a Dynkin diagram of type $\mathrm E_8, \mathrm E_7, \mathrm E_6$, and $\mathrm D_{4+n}$, respectively.  These fiber types play an important role in the study of quadratic base change of elliptic surfaces. We record the following useful observation for later use:
\begin{lemma}\label{lemma_star-1comp}
   In each star fiber $F$, there are exactly four components with odd multiplicity, and they are disjoint.
   \end{lemma}
   \begin{proof}
       This can be directly verified from \cite[Theorem 5.12]{SchShi19}.
   \end{proof}

The canonical bundle of an elliptic surface $f:S\to C$ is given by
$$\omega_{S} := f^*(\omega_C\otimes \mathbb L_S),$$
where $\mathbb L_S = f_*(\omega_{S/C}) $ is the {\it fundamental line bundle} of degree $\chi(S)$ on $C$.

\begin{lemma}\label{lemma_e(S)} The degree of fundamental line bundle of an elliptic surface satisfies 
 $$\deg \mathbb L_S = \frac{1}{12}e(S) = \frac{1}{12}\sum_{c\in \Delta} e(f^{-1}(c)),$$
 where $e(\cdot)$ denotes the topological Euler characteristic, and $\Delta$ is the subset of $C$ that corresponds to the singular fibers.
\end{lemma}
In particular, $\deg\mathbb L_S>0$ if and only if there is at least one singular fiber.

\begin{lemma}
    The geometric genus $p_g(S)$ of an elliptic surface satisfies $$p_g(S)=\deg \mathbb L_S+g(C)-1.$$
\end{lemma}

\begin{remark}\normalfont\label{remark_pg=1}
 There are exactly two types of elliptic surfaces $S\to C$ with $p_g=1$ that have positive Euler characteristic: one has $g(C)=0$ and $\deg\mathbb L_S=2$, and is an elliptic K3 surface; the other has $g(C)=1$ and $\deg\mathbb L_S=1$, and is called an elliptic-elliptic surface.
\end{remark}
An elliptic-elliptic surface $S\to C$ is an elliptic surface with $p_g=q=1$ together with a section \cite{EGW,GZ}. Equivalently, the base curve has genus one, and the degree of the fundamental line bundle is one also. The cohomology group $H^2(S,\Z)$ carries a weight two Hodge structure
with Hodge numbers $h^{2,0}=h^{0,2}=1, h^{1,1}=12$.
Therefore, $S$ has a unique nonzero holomorphic 2-form $\omega$ up to a scaling. It vanishes along a unique fiber, which is called the canonical fiber, denoted as $K_S$. 

We close this preliminary subsection with a useful lemma regarding torsion in the middle cohomology of an elliptic surface.

\begin{lemma}\cite[Lemma D]{GurjarShastri1985}\label{lemma:GS}  If an elliptic surface $S$ has a section and at least one singular fiber, then its second cohomology $H^2(S,\Z)$ is torsion-free.
\end{lemma}


\subsection{Shioda--Inose structure}
Shioda and Inose \cite{ShiIno77} showed that every  K3 surface with maximal Picard rank has an associated pair of isogenous CM elliptic curves. More precisely, for a Picard maximal K3 surface $Y$, the transcendental lattice is even integral positive definite binary quadratic form $$T_Y\cong \begin{bmatrix}
    2a&b\\b&2c
\end{bmatrix};\ a,b,c\in \Z$$
with discriminant $d=4ac-b^2$. 
\begin{theorem}\label{thm_singularK3}
    A  K3 surface with maximal Picard number is uniquely determined by its transcendental lattice.
\end{theorem}
There is a pair of isogenous elliptic curves $E$ and $E'$ associated with periods $\tau=\frac{-b\pm\sqrt{-d}}{2a}$. Shioda and Inose showed that there is a Hodge isometry 
$$\Phi:T_Y\cong T_{E\times E'}$$
which is geometrically realized by the diagram below.
\begin{figure}[ht]
    \centering
\begin{equation}\label{diagram_ShiodaInose}
\begin{tikzcd}
Y\arrow[dr,dashed,"\pi_1"] && E\times E'\arrow[dl,dashed,"\pi_2"']\\
&\Kum(E\times E').
\end{tikzcd}
\end{equation}
\end{figure}

\noindent Here, $X=\Kum(E\times E')$ is the Kummer surface associated with $E\times E'$, and $\pi_i$ is a rational double cover, such that $(\pi_1)_*:T_Y(2)\cong T_{X}$. Consequently, they show that the Hodge isometry $\Phi$ is induced by an algebraic cycle.

To construct $\pi_1$, Shioda--Inose find a special elliptic fibration on the Kummer surface and realize $Y$ via a quadratic base change:
\begin{lemma}\textup{(\cite{ShiIno77})}\label{lemma_ShiodaInose} The Kummer surface $X=\Kum(E_1\times E_2)$ associated with product of two elliptic curves admits an elliptic fibration $X\to \mathbb P^1$ with a section and singular fiber types (i) $II^*, IV^*,I_0^*$, or (ii) $II^*,I_{b_1}^*,I_{b_2}^*$, $b_1+b_2\le 2$. Moreover, if $$\pi: Y\dashrightarrow X$$ is the rational double cover arising from quadratic base change branched at the two star fibers different from $II^*$, then $Y$ is a K3 surface and $\pi_*$ induces a Hodge isometry of transcendental lattices
$$\pi_*:T_{Y}(2)\cong T_X.$$
\end{lemma}

In particular, the K3 surface $Y$ admits a Nikulin involution with 8 fixed points. In general, Morrison \cite{Morrison84} showed that if a K3 surface $Y$ admits a transcendental Hodge isometry $T_Y\cong T_A$ to the transcendental lattice of an abelian surface $A$, then there is a diagram similar to \eqref{diagram_ShiodaInose}, but with $E\times E'$ replaced by $A$; such a diagram is called a \textit{Shioda--Inose structure} for the K3 surface $Y$.
In this paper, we explore an analogous diagram and transcendental lattice comparison between K3 surfaces and elliptic-elliptic surfaces.





\section{Invariant cycle theorem for $H^1$}\label{sec_H^1}

In this section, we show that the local invariant cycle theorem holds over $\Z$ for $H^1$ in any projective semistable family.
\begin{theorem}\label{ssh1}
    Let $f:\mathcal X \to \Delta$ be a projective semistable family of varieties. Then the restriction map
    \[r:H^1(\mathcal{X},\Z)\to H^1(X_t,\Z)^{\pi_1(\Delta^*)}\]
    is surjective.
\end{theorem}
The structure of the argument is as follows. First, we prove it for certain families of abelian varieties coming from the Mumford construction. This implies the statement for families of curves, via the relative Jacobian. Lastly, using a relative Lefschetz hyperplane argument, the case of arbitrary varieties follows from the case of curves.
\subsection{Abelian varieties}
Let $\cA \to \Delta$ be a projective semistable family such that $A_t$ is an abelian variety of dimension $g$ for $t\neq 0$. The monodromy action on $H^1(A_t,\Q)$ is unipotent.
Since the total space $\cA$ is smooth and the central fiber $A_0$ is reduced, it is possible to choose a section, which gives $\cA^*\to \Delta^*$ the structure of an abelian scheme. 
The N\'{e}ron model $Ner(\mathcal A) \to \Delta$ is a group scheme over $\Delta$ extending $\cA^*$. By the inertia monodromy theorem of Grothendieck, the connected component of the central fiber $Ner(\cA)_0$ is a semiabelian variety, which we denote by  $A_0^\circ$.
\begin{definition}
    A semiabelian variety $A_0^\circ$ is an extension of algebraic groups
    \begin{equation}\label{semiabelian}
    0 \to \mathbb T \to A_0^\circ \to B_0 \to 0,\end{equation}
    where $\mathbb T\simeq (\C^\times)^r$ is a torus of rank $r$, and $B_0$ is an abelian variety of dimension $g-r$.
\end{definition}
Taking $H^1(-,\Z)$ of (\ref{semiabelian}) above, we obtain the short exact sequence
\begin{equation}\label{extensionH1}
0 \to H^1(B_0,\Z) \to H^1(A_0^\circ,\Z) \to H^1(\mathbb T,\Z) \to 0.\end{equation}
This is an extension of free abelian groups, and is therefore split, so we have:
$$H^1(A_0^\circ,\Z) \simeq \Z^r \oplus \Z^{2g-2r} \simeq \Z^{2g-r}.$$
The degeneration $\cA \to \Delta$ gives rise to a mixed Hodge structure on $H^1(A_t,\Z)$, where the weight filtration is defined in terms of the nilpotent operator $T-I$; see \cite{Deligne-HodgeIII}, \cite{cattani}.
$$W_{-1}=0 \subset \mathrm{im}(T-I)_{\sat} \subset \ker(T-I) \subset H^1(A_t,\Z)=W_2.$$
The sequence (\ref{extensionH1}) above comes from the weight filtration as follows:
\[ 0 \to W_1/W_0 \to W_2/W_0 \to W_2/W_1\to 0. \]
Observe also that $\mathrm{im}(T-I)\simeq H^1(A_t,\Z)/\ker(T-I)$, so we get a finite index containment:
\begin{equation}\label{weightgr} W_0 \supset \mathrm{im}(T-I)\simeq  W_2/W_1 \simeq H^1(\mathbb T,\Z)= X^*(\mathbb T), \end{equation}
the character lattice of the algebraic torus. Dualizing, we have a finite index containment
$$W_0^\vee \subset X_*(\mathbb T).$$

The central fiber $A_0$ of the semistable family is an SNC variety. Its cohomology can be computed from the Mayer--Vietoris spectral sequence associated to a covering by component neighborhoods. Let $K = K(A_0)$ denote the dual complex of $A_0$, the simplicial complex encoding the SNC incidence of the irreducible components. For $H^1(A_0,\Z)$, we obtain the following short exact sequence:
\begin{equation}\label{mvss} 0 \to H^1(K,\Z) \to H^1(A_0,\Z) \to H^0(K,\mathcal H^1) \to 0,\end{equation}
where $\mathcal H^1$ is the sheaf that assigns to each stratum $\sigma$ in $K$ the cohomology $H^1(A_\sigma,\Z)$ of the corresponding SNC stratum in $A_0$, with the induced restriction maps on $H^1$. The homotopy type of $K$ depends only on the punctured family $\cA^* \to \Delta^*$ \cite{nicaise-xu}, and so using the Mumford construction in the next subsection, we will see that $K$ must have the homotopy type of the real torus $\R^r/\Z^r$. On the other hand, $H^0(K,\cH^1)$ is more subtle in general, but we will describe it explicitly, in the case of a Mumford model, as being isomorphic to $H^1(B_0,\Z)$.



\subsection{Mumford Construction}
The Mumford construction gives a particularly nice semistable model for the degeneration $\cA \to \Delta$, which is determined from the mixed Hodge structure data above, plus a choice of fan. For a modern reference on this, we refer to \cite{engel-degaayfortman-schreieder}.

Let $N= X_*(\mathbb T)$ be the cocharacter lattice of $\mathbb T$, and let $N' = W_0^\vee\subset N$ the finite index sublattice from \ref{weightgr}. Choose a rational, polyhedral, unimodular fan decomposition $\Sigma$ of $N_\R$ that is $N'$-{\it periodic}, meaning that it is preserved under the translation action of $N'$. In \cite{mumford}, \cite{faltings-chai}, it is shown that there exists a smooth variety $\cA^\Sigma$, which is an alternative semistable model for $\mathcal A \to \Delta$. Its central fiber $A^\Sigma_0$ is an SNC union of components, and the normalization of each component is a locally trivial family of toric varieties over $B_0$. If, additionally, the fan is {\it polarizable}, meaning that there exists an $N'$-periodic tiling of $N_\R$ by polytopes dual to $\Sigma$, then $\cA^\Sigma \to \Delta$ is a projective family. A polarizable fan $\Sigma$ always exists if the original family $\cA\to \Delta$ was projective, using the Voronoi fan and the dual Delaunay tiling.

More precisely, choose a polyhedral fundamental domain $\cF$ for the action of $N'$ on the fan $\Sigma$. For each vertex $v$ of $\Sigma$ in $\cF$, there is an irreducible component of $A_0^{\Sigma}$ whose normalization is a locally trivial family of toric varieties over $B_0$ with fiber equal to the toric variety associated to the fan $\mathrm{Star}(v)$. These $[N:N']$ irreducible components are glued together according to the incidence combinatorics of $\Sigma$ modulo $N'$, but also with translations by elements of $B_0$. Any translation on $B_0$ is deformation equivalent to the identity, and hence it will not matter for the topology of $A_0^{\Sigma}$ as a stratified space. Nonetheless, we denote by ${A'}_0^{\Sigma}$ the SNC variety homeomorphic to $A_0^\Sigma$ in which all the gluing translations in $B_0$ are trivial. In particular, we have a well-defined map ${A'}_0^{\Sigma} \to B_0$.
There is also a toric moment map:
\[\mu: {A'}_0^{\Sigma} \to N_\R/N' \times B_0.\]
The fibers of $\mu$ are real tori, generically rank $r$, but with drops in dimension according to the depth of the stratum in the polytope tiling of $N_\R/N'$. The Clemens collapse map $A_t \to A_0^{\Sigma}$ is perfectly complementary to $\mu$; it is an isomorphism on the generic open stratum, with real torus fibers of rank $c$ over strata of codimension $c$ in $A_0^{\Sigma}$.
As a result, the composition of the Clemens collapse with $B_0$-translation and the moment map gives a real torus fiber bundle:
$$\tau: A_t \to A_0^{\Sigma} \overset{\sim}\to {A'}_0^{\Sigma} \to N_\R/N'\times B_0,$$
with fibers isomorphic to $\R^r/\Z^r$. This is the tropicalization of the family.
\begin{lemma}\label{lictmumford}
    The local invariant cycle theorem for $H^1$ with coefficients in $\Z$ holds for any projective Mumford construction $\cA^{\Sigma} \to \Delta$.
\end{lemma}
\begin{proof}
    The Leray-Serre spectral sequence applied to the torus bundle $\tau: A_t \to N_\R/N' \times B_0$ has $E_2$-page
    \[H^p(N_\R/N'\times B_0,\, R^q\tau_*\Z) \Rightarrow H^{p+q}(A_t,\Z).\]
    For $p+q=1$, this results in the four-term exact sequence
$$0 \to H^1(N_\R/N'\times B_0,\,\Z) \to H^1(A_t,\Z) \to H^0(N_\R/N'\times B_0,\, R^1\tau_*\Z) \overset{d_2}\to H^2(N_\R/N'\times B_0,\,\Z).$$
The stalks of the local system $R^1\tau_*\Z$ are free of rank $r$, so its group of global sections must be free of rank $\leq r$. The first two terms are free of ranks $2g-r$ and $2g$, respectively, so this implies that $d_2=0$, and in fact $R^1\tau_*\Z$ is a trivial local system. Therefore, we have a short exact sequence of free abelian groups:
\begin{equation}\label{leray}
    0 \to H^1(N_\R/N' \times B_0,\Z) \to H^1(A_t,\Z) \to H^1(\tau^{-1}(*),\Z)\to 0.
\end{equation}
Next, we take up the short exact sequence (\ref{mvss}) for the Mumford degeneration. Every stratum in the SNC central fiber $A_0^{\Sigma}$ has Albanese isomorphic to $B_0$, and the restriction maps on $H^1$ give compatible isomorphisms to $H^1(B_0,\Z)$. The dual complex $K = K(A_0^{\Sigma})$ is homeomorphic to the real torus $N_\R/N'$. Therefore, (\ref{mvss}) splits, and we have
$$H^1(A_0^{\Sigma},\Z) \simeq H^1(N_\R/N',\Z) \oplus H^1(B_0,\Z).$$

On the other hand, $N'=W_0^\vee$ and $H_1(N_\R/N',\Z)\simeq N'$, so we get canonical isomorphisms
\[H^1(N_\R/N',\Z)\simeq W_0,\,\,\,  H^1(B_0,\Z)\simeq \mathrm{gr}_1^W H^1(A_t,\Z).\]
Hence, the short exact sequence (\ref{leray}) is isomorphic to
\[0 \to \ker(T-I) \to H^1(A_t,\Z) \to \mathrm{gr}_2^W H^1(A_t,\Z)\to 0.\]
Therefore, we actually have the exact sequence
$$0 \to H^1(A_0^{\Sigma},\Z) \to H^1(A_t,\Z) \overset{T-I}\to H^1(A_t,\Z),$$
as desired for the LICT$_\Z$ on $H^1$.
\end{proof}

\subsection{Curve degenerations}
Let $\cC \to \Delta$ be a projective semistable degeneration of curves of genus $g$. The central fiber $C_0$ is a nodal curve of arithmetic genus $g$, with dual graph $\Gamma(C_0)$.
Applying the relative $\Pic^0$-functor gives an abelian group scheme $\mathcal A$ over $\Delta$, projective over the generic point. The fiber over $0$ is a semiabelian variety, as in (\ref{semiabelian}), with abelian part $B_0$ isomorphic to $\Pic^0(C_0^\nu)$, the Jacobian of the normalization of $C_0$, and torus rank $r$ is equal to $h_1(\Gamma(C_0))$.

\begin{theorem}\label{thm_LICTZ-curves}
    The local invariant cycle theorem holds over $\Z$ for any semistable degeneration of curves $\cC\to \Delta$.
\end{theorem}
\begin{proof}
    Applying the Mumford construction, one can produce a projective semistable degeneration of abelian varieties $\cA^{\Sigma}$ extending the relative Jacobian $\cA \to \Delta$, for which the LICT$_\Z$ holds by Lemma \ref{lictmumford}.

    Next, choose any section of $\cC \to \Delta$, and define the relative Abel-Jacobi map using this section, away from the fiber at 0:
    \[ \cC \dasharrow \cA^{\Sigma}. \]
This rational map can be resolved by blowing up points in the central fiber, which introduces rational components to $C_0$. Crucially, the blow up operation has no effect on the homotopy type of the dual graph $\Gamma_0$, and also no effect on $\Pic^0(C_0^\nu)$. The pullback via the resolved Abel-Jacobi map $\alpha$ gives an isomorphism 
$$H^1(A_t,\Z) \to H^1(C_t,\Z),$$
which commutes with the monodromy $T-I$, as well as vertical isomorphisms
\[\xymatrix{
0 \ar[r]& H^1(N_\R/N',\Z) \ar[d]\ar[r]& H^1(A_0^{\Sigma},\Z) \ar[d]\ar[r]& H^1(B_0,\Z)\ar[d] \ar[r]& 0 \\
0 \ar[r]& H^1(\Gamma(C_0),\Z) \ar[r]& H^1(C_0,\Z) \ar[r]& H^1(C_0^\nu,\Z) \ar[r]& 0.
}\]
Therefore, the LICT$_\Z$ holds on $H^1$ for $\mathcal C \to \Delta$; it holds trivially on $H^0$, $H^2$.
\end{proof}


\subsection{Reduction to curves}
Lastly, the question of LICT$_\Z$ on $H^1$ for an arbitrary semistable family can be reduced to the case of curves. Let $\cX \to \Delta$ be a projective semistable family of relative dimension $n\geq 2$. Let $\cY \subset \cX \to \Delta$ be a general hyperplane section of the family, which is also semistable. By the Lefschetz hyperplane theorem, the restriction map $H^1(X_t,\Z) \to H^1(Y_t,\Z)$ is injective. Furthermore, it commutes with monodromy:
\[\xymatrix{
H^1(X_t,\Z)\ar[d] \ar[r]^{T} & H^1(X_t,\Z)\ar[d] \\
H^1(Y_t,\Z) \ar[r]^{T} & H^1(Y_t,\Z).
}\]
This is part of the Clemens--Schmid sequence by adding in a column on the left with the restriction map $H^1(X_0,\Z)\to H^1(Y_0,\Z)$ on the central fibers. We prove a version of the Lefschetz hyperplane theorem (see Theorem \ref{thm_reduction-hyp}) for LICT$_\Z$ in cohomological degrees $k=1$. By induction, this reduces the proof for $k=1$ to the case of curves.

\begin{theorem}[Lefschetz hyperplane theorem for LICT$_{\Z}$]\label{thm_reduction-hyp}
  Let $f:\cX \to \Delta$ be a projective semistable family of relative dimension $n>1$. Let $g:\cY \subset \cX \to \Delta$ be a general hyperplane section family. If the \textup{LICT}$_{\Z}$ holds for $g$ in degree $1$, then it also holds for $f$ in degree $1$.
\end{theorem}

\begin{proof}
We have the following commutative diagram:
 \begin{figure}[ht]
    \centering
\begin{equation*}
\begin{tikzcd}
0\arrow[r]&H^1(X_0,\Z)\arrow[r,"r"]\arrow[d,hookrightarrow,"\iota^*"]& H^1(X_t,\Z)\arrow[r,"T-I"]\arrow[d,hookrightarrow,"\iota^*"]&H^1(X_t,\Z)\arrow[d,hookrightarrow,"\iota^*"]\\
0\arrow[r]&H^1(Y_0,\Z)\arrow[d,"\delta"]\arrow[r,"r"]& H^1(Y_t,\Z)\arrow[r,"T-I"]&H^1(Y_t,\Z)\\
& H^{2}(X_0,Y_0,\Z).
\end{tikzcd}
\end{equation*}
\end{figure}

The maps labeled $r$ are both injective by the Clemens-Schmid sequence.
 The  three vertical maps labeled $\iota^*$ all injective by Lemma \ref{lemma_LefHyp-coh} and the usual Lefschetz hyperplane theorem.    Observe that any monodromy invariant class $\alpha\in H^1(X_t,\Z)^{\pi_1(\Delta^*)}$ restricts to an invariant class $\beta=\iota^*\alpha$ in $H^1(Y_t,\Z)^{\pi_1(\Delta^*)}$. Since we are assuming that LICT$_\Z$ holds for $g$,  $\beta$ lifts to $\beta_0\in H^1(Y_0,\Z)$. We claim $\beta_0$ lies in the image of $\iota^*: H^1(X_0, \mathbb{Z}) \to H^1(Y_0, \mathbb{Z})$.

By LICT$_\Q$, there exists a class $\alpha_0\in H^1(X_0,\Z)$ and an integer $m$, such that $r(\alpha_0)=m\alpha$. By commutativity of the diagram,
    \begin{equation}\label{eqn_res-to-hyp}
    r\circ\iota^*(\alpha_0)=m\beta.
    \end{equation} 
Therefore, since $r$ is injective, 
    $$\iota^*(\alpha_0)=m\beta_0$$
This implies that the image $m\delta(\beta_0)=0$. However, Lemma \ref{lemma_LefHyp-coh} implies that $H^2(X_0, Y_0;\Z)$ is torsion free. Therefore, $\delta(\beta_0)=0$, and so it lies in the image of $\iota^*$ as claimed. Writing $\beta_0 = \iota^*\alpha_0'$, we
can see that $r(\alpha_0') =\alpha$ because $r(\beta_0)=\beta$.

\end{proof}

\noindent\textit{Proof of Theorem \ref{ssh1}.}
This now follows from Theorems \ref{thm_LICTZ-curves} and \ref{thm_reduction-hyp}.\qed

\section{Invariant cycle theorem for $H^2$}\label{sec_postiveresults}
In this section, we discuss some positive results on Question \ref{question_ICTZ} for $H^2$.

\subsection{Slicing the Clemens--Schmid sequence}
The Clemens--Schmid sequence is a concatenation of the Wang sequence and the long exact sequence of the pair $(\mathcal{X},\mathcal{X^*})$ \cite[Corollary 11.44]{PS08}, both of which are exact over $\Z$:

   \begin{figure}[ht]
    \centering
\begin{equation}\label{diagram_ClemensSchmid}
\begin{tikzcd}
H^k(\mathcal{X}) \arrow[rr] \arrow[dr,"r"]&& H^k(X_t)\arrow[r,"T-I"] & H^k(X_t)\\
&H^k(\mathcal{X}^*)\arrow[ur,"s"]\arrow[dr,"\delta"]\\
H^{k-1}(X_t)\arrow[ur]&& H^{k+1}(\mathcal{X},\mathcal{X}^*)\arrow[r]& H^{k+1}(\mathcal{X}).
\end{tikzcd}
\end{equation}
\end{figure}


Note that we have used $T-I$ instead of $N=\log(T)$ because $N$ is defined over $\bbQ$, but not over $\Z$ in general. Nevertheless, the first horizontal sequence is still exact over $\bbQ$, due to the fact that $\ker(N)=\ker(T-I)$; see the proof of \cite[Theorem 11.43]{PS08}. 
It makes sense to ask whether the Clemens--Schmid sequence remains exact over $\Z$. Exactness at the term $H^n(X_t)$ is equivalent to Question \ref{question_ICTZ}.


\begin{proposition}\label{prop_torsionX0}
    Let $f:\mathcal{X}\to \Delta$ be a semistable families of complex projective varieties of relative dimension $n$, with $\mathcal{X}$ smooth. If (i) $H^{k-1}(X_t,\bbQ)=0$ and (ii) $H_{2n-k+1}(X_0,\Z)$ is torsion-free, then \textup{LICT}$_\Z$ holds for $f$ in degree $k$. In particular, the Clemens--Schmid sequence (the top row of \eqref{diagram_ClemensSchmid}) is exact over $\Z$.
\end{proposition}

\begin{proof}
The assumption (i) implies $H^{k-1}(X_t,\Z)$ is torsion; The assumption (ii) is equivalent to the torsion-freeness of $H^{k+1}(\mathcal X,\mathcal X^*,\Z)$: This follows from isomorphisms 
    $$H^{k+1}(\mathcal{X},\mathcal{X}^*,\Z)\cong H^{k+1}(\mathcal{X},\partial \mathcal{X},\Z)\cong H_{2n-k+1}(\mathcal{X},\Z)\cong H_{2n-k+1}(X_0,\Z),$$
which arises from excision, Lefschetz duality, and the fact that there is a deformation retract of $\mathcal{X}$ onto $\mathcal{X}_0$ (cf. \cite{Cle77}, \cite[Proposition C.11]{PS08}).

    
  Given any monodromy invariant class $\alpha\in \ker(T-I)$, it lifts to $\tilde{\alpha}\in H^k(\mathcal X^*,\Z)$. By the exactness of the Clemens--Schmid sequence over $\bbQ$ (LICT$_{\bbQ}$), 
there exist $\beta\in H^k(\mathcal{X},\Z)$ and an integer $m\in \Z$ such that $s\circ r(\beta)=m\alpha$. In particular, the difference $r(\beta)-m\tilde{\alpha}\in H^k(\mathcal X^*,\Z)$ is torsion, by the assumption (i) that the previous term in the exact sequence $H^{k-1}(X_t,\Z)$ is torsion. Then applying $\delta$ and using the assumption (ii), we obtain that $\delta(m\tilde{\alpha})=0$, so $\delta(\tilde{\alpha})=0$. Hence $\tilde{\alpha}$ lifts to $H^k(\mathcal X,\Z)$. The claim follows.
\end{proof}

The above generalizes the argument in \cite[Lemma 3.5]{Fri84} and \cite[Theorem 4.1]{FriGri24} to show the exactness of the Clemens--Schmid sequence over $\Z$ for a semistable family of K3 surfaces and I-surfaces. In particular, we prove the following.

\begin{theorem}\label{thm_H^2}
     Let $f:\mathcal{X}\to \Delta$ be a semistable degeneration of complex projective varieties with $q(X_t)=0$. Then the \textup{LICT}$_{\Z}$ holds for $f$ in cohomological degree 2.
\end{theorem}
\begin{proof}
  The assumption $q(X_t)=0$ implies that $H^1(X_t,\bbQ)=0$. We apply $n=2$ to Proposition \ref{prop_torsionX0}. It remains to show that $H_{2n-1}(X_0,\Z)$ is torsion-free, where $\dim(X_0)=\dim(X_t)=n$. By the universal coefficient theorem, $H_{2n-1}(X_0,\Z)_{tor}\cong H^{2n}(X_0,\Z)_{tor}$, which is zero by Lemma \ref{lemma_degMV}.
\end{proof}

\subsection{Global invariant cycle theorem over 1-dimensional base} Theorems \ref{ssh1} and \ref{thm_H^2} also imply the following
\begin{corollary}\label{cor_GICT_H^1}
  Let $f:\mathcal X\to C$ be a projective semistable family over a smooth projective curve. Then the \textup{GICT}$_{\Z}$ holds for $f$ in degree $1$, and moreover holds in degree $2$ if $q(X_t)=0$. 
\end{corollary}
\begin{proof}
    Let $\mathcal X_U\to U$ be the maximal smooth subfamily. Then a monodromy invariant class $\alpha\in H^1(X_t,\Z)^{\pi_1(U,t)}$ lifts to $\alpha_U\in H^1(\mathcal X_U,\Z)$. For each of the punctures $p\in C\setminus U$, an analytic neighborhood $\Delta_p$ supports the semistable family $\mathcal X_{\Delta_p}\to \Delta_p$. We can assume $t\in \Delta_p$, and by Theorem \ref{ssh1}, the invariant class lifts to $\alpha_p\in H^1(\mathcal X_{\Delta_p},\Z)$. Their restriction to $H^1(X_t,\Z)$ coincides, so their restriction to annulus family $\mathcal X_{\Delta_p^*}$ differs by a class $\alpha_U-\alpha_p\in \ker(H^1(\mathcal X_{\Delta_p^*},\Z)\to H^1(X_t,\Z))$. By the Wang sequence \eqref{diagram_ClemensSchmid}, this class comes from $f^*(\eta_p)$, where $\eta_p\in H^1(\Delta_p^*,\Z)$. Therefore, if we adjust $\alpha_p$ by $f^*(\eta_p)$, the two classes agree over the annulus, and by Mayer--Vietoris sequence, lift to a class on $H^1(\mathcal X_U\cup \mathcal X|_{\Delta_p},\Z)$. Repeating this process for each of the points in $C\setminus U$, we get a class $\alpha_{\mathcal X}$ in $H^1(\mathcal X,\Z)$ whose restriction to $X_t$ is $\alpha$. The claim follows. The same argument applies for $H^2$ using Theorem \ref{thm_H^2}.
\end{proof}

\section{Quadratic base change and variations of Hodge structure}\label{sec_constantHodge}

In this section, we investigate certain special base changes of elliptic surfaces, in which the degree of the fundamental line bundle is smaller than expected, and study their Hodge structures. Our main application is to construct elliptic-elliptic surfaces as rational double covers of elliptic K3 surfaces. 
\subsection{Euler defect of base change}

Let $X^*\to \Delta^*$ be a family of smooth elliptic curves parameterized by a punctured disk $\Delta^*$, and $X\to \Delta$ be the Kodaira--N\'eron model with singular fiber $F$ over $0$. Let $\tilde{\Delta}\to \Delta$ be the $2$-to-1 cover totally branched at $0$, then the base change $X\times_{\Delta}\tilde{\Delta}$ is birational to a new family of elliptic curves $\tilde{X}\to \tilde{\Delta}$, with Kodaira fiber type $F'$. Using Kodaira's notation, the base change fiber type table is listed below \cite[p.105]{SchShi19} for star fibers:

\begin{table}[!h]
\begin{center}
\begin{tabular}{|| c |c  ||}
\hline
$F$ &  $F'$ \\
 \hline\hline
$II^*$ & $IV^*$\\  
 \hline
$III^*$ & $I_0^*$\\
 \hline
$IV^*$ & $IV$\\
 \hline
$I_n^*$ & $I_{2n}$\\
 \hline
\end{tabular}
 \caption{Double cover and reduction of Kodaira's star fibers}
 \label{table_star}
\end{center}
\end{table}

\begin{definition}\normalfont
    For each Kodaira singular fiber of type $F$, we define 
    \begin{equation}\label{eqn_delta(F)}
        \delta(F):=\frac{1}{12}(2\cdot e(F)-e(F')),
    \end{equation}
with respect to a quadratic base change, where $e(\cdot)$ is the Euler characteristic.

\end{definition}
\begin{proposition}\label{prop_delta=0,1}
    For every Kodaira fiber of type $F$, $\delta(F)\in \{0,1\}$. We have $\delta(F)=1$ if and only if $F$ is a star fiber.
\end{proposition}
\begin{proof}
   This is a direct consequence of Table \ref{table_star} and the fact that $e(II)=2$, $e(III)=3$, $e(IV)=4$, $e(II^*)=10$, $e(III^*)=9$, $e(IV^*)=8$, $e(I_n^*)=6+n$, $e(I_n)=n$.
\end{proof}
 Let $f:X\to C$ be an elliptic surface and $\pi:C'\to C$ a branched double cover. Let $f':X'\to C'$ be the relatively minimal regular model of the fibered product $X\times_CC'$. 
 Here, we investigate the relationship between the degree of fundamental line bundles $d=\deg(\mathbb L_X)$ and $d'=\deg(\mathbb L_{X'})$.

\begin{lemma}\label{lemma_Euler-defect}
  There is an inequality 
$$d'\le 2d.$$
Moreover, the difference 
\begin{equation}\label{eqn_delta-global}
    \delta:=2d-d'
\end{equation}
equals the number of points $p\in C$ such that
    \begin{itemize}
        \item the fiber $f^{-1}(p)$ is a star fiber, and 
        \item $p$ is a branch point of the double cover $C'\to C$.
    \end{itemize}
\end{lemma}
\begin{proof}
    By Lemma \ref{lemma_e(S)}, $d=e(X)/12$, and $d'=e(X')/12$, and both $e(X)$ and $e(X')$ are determined by the Euler characteristic of the singular fibers. Hence, it suffices to compare the singular fibers on the two surfaces. When $p\in C$ corresponds to a singular fiber $F$ of $f$, then by Proposition \ref{prop_delta=0,1}, there are three possibilities
    \begin{itemize}
        \item[(i)] $p$ is unbranched for $\pi$, then it corresponds to two singular fibers on $X'$ of type $F_p$;
        \item[(ii)] $p$ is a branch point, and $F_p$ is not of star type, and $\delta(F_p)=0$;
        \item[(iii)] $p$ is a branch point, and $F_p$ is of star type, and $\delta(F_p)=1$.
    \end{itemize}
   Only the last case contributes to the difference $2e(X)-e(X')$. Using Lemma \ref{lemma_e(S)} again, we find 
   $$2e(X)-e(X')=\sum_{p\in \Delta}\big (2e(F_p)-e(F_p')\big ),$$
where $\Delta\subset C$ is the collection of points satisfying (iii), and $F_p'$ is the type of the singular fiber in $X'$ according to Table \ref{table_star}. After dividing the equation by 12, the result follows from the definition of $\delta$ in \eqref{eqn_delta(F)}.
\end{proof}



\begin{definition}\normalfont
    We call $\delta$ (cf. \eqref{eqn_delta-global}) the \textit{Euler defect} of the quadratic base change of elliptic surface $X\to C$ with respect to a double cover $C'\to C$. We say the base change is \textit{defective} if $\delta>0$.
\end{definition}
\begin{remark}\normalfont
    The Euler defect can be defined in general for higher degree base change. This will be studied more systematically in a subsequent paper. In this paper, we focus on the case of quadratic base change, when $g(C')=0,1$.
\end{remark}
\subsection{Defective double cover of an elliptic K3}
Now we start with an elliptic K3 surface $f:X\to \mathbb P^1$  with (a certain number of) star fibers. We denote by $f':X'\to C'$ the relatively minimal regular model of the elliptic surface that comes from the quadratic base change $C'\to \mathbb P^1$. The geometry of the surface $X'$ is listed below for $g(C')=0,1$ and $\delta=2,3,4$ below.

\begin{corollary}\label{cor_ell-ell-by-base-change}
\-
    \begin{itemize}
        \item[(1)] ($\delta=2$, K3 surface) If $\mathbb P^1=C'\to \mathbb P^1$ is the double cover branched at points lying under two star fibers, then $f':X'\to \mathbb P^1$ is an elliptic K3 surface. 
        \item[(2)] ($\delta=3$, elliptic-elliptic surface) If $E=C'\to \mathbb P^1$ is the double cover branched at four points with three points lying under star fibers, then $f':X'\to E$ is an elliptic-elliptic surface. 
        \item[(3)] ($\delta=4$, abelian surface) If $E=C'\to \mathbb P^1$ is the double cover branched at four points that all lie under star fibers, then $f':X'\to E$ is a trivial family, so $X'$ is isomorphic to a product of elliptic curves. 
    \end{itemize}
\end{corollary}
\begin{proof}
By Lemma \ref{lemma_Euler-defect}, the Euler defect of the three quadratic base change are $\delta=2,3,4$, respectively. 
The elliptic K3 surface $f:X\to \mathbb P^1$ has degree of fundamental line bundle $d=2$. Therefore, by Lemma \ref{lemma_Euler-defect}, $d'=2\cdot 2-\delta$, which equals to $2,1$, and $0$, respectively. In case (1), the base curve is $\mathbb P^1$ and $d'=2$, so $X'$ is a K3 surface. In case (2), the base curve has genus one and $d'=1$, so $X'$ is an elliptic-elliptic surface. In case (3), $d'=0$ implies that $X'\to C'$ is isotrivial with all fibers smooth. In fact, it is trivial because the only possible configuration of four star fibers on a K3 surface is $4\times I_0^*$. The local monodromy on $H^1$ near each singular fiber has order 2, so the double cover $C' \to \PP^1$ branched at those four critical values trivializes the monodromy representation globally.
\end{proof}

Note that case (1) includes the Shioda--Inose double covers (Lemma \ref{lemma_ShiodaInose}), together with more general double covers studied in \cite{Mehran07}. Case (3) is the standard rational double cover $E\times E'\dashrightarrow \Kum(E\times E')$ of the Kummer surface for a product of two elliptic curves.
\begin{remark}\normalfont
The constructions in cases (1) and (3) are rigid with respect to $X\to \mathbb P^1$, while in the case (2), the construction of elliptic-elliptic surfaces has a free parameter, since one can vary the fourth branch fiber. The variation of Hodge structure for such a one-parameter family will be studied in Section \ref{sec_Hodge}.
\end{remark}

\begin{remark}\normalfont
   The assumptions in Corollary \ref{cor_ell-ell-by-base-change} put a lower bound on the Picard number $\rho=\textup{rank}(NS(X))$ of the K3 surface $X$. Indeed, case (1), (2), and (3) assume $X\to \mathbb P^1$ has at least 2, 3, and 4 star fibers. A star fiber has at least $5$ components (realized by $I_0^*$), so by the Shioda--Tate formula \cite[Corollary 6.7]{SchShi19}, $\rho$ is at least $10,14$, and $18$, respectively. This gives a lower bound on Picard number of $X'$ as well (see Lemma \ref{lemma_TXpullback}).
\end{remark}

 \subsection{Birational geometry of defective double cover}
In this section, we analyze defective double covers in more detail from the point of view of birational geometry.
\begin{lemma}\label{lemma_normalization}
    Suppose that the fiber $F$ over 0 of a Kodaira--N\'eron model $X\to\Delta$ is a star fiber. Let $\tilde\Delta\to \Delta$ be a double cover branched at $0$. Then the normalization $\tilde{X}'$ of $X\times_{\Delta}\tilde{\Delta}$ is smooth, and contains four disjoint $(-1)$ curves. After contracting the four curves, one obtains the relatively minimal model $X'\to \tilde\Delta$ with singular fiber of type $F'$ listed in Table \ref{table_star}. 
\begin{proof}
   According to \cite[p.125]{HM98}, the composition $p:\tilde{X}'\to X\times_{\Delta}\tilde{\Delta}\to X$ is the double cover branched along the components with odd multiplicity on the fiber $F$. The result follows from the Lemma \ref{lemma_star-1comp}.
\end{proof}

\end{lemma}

Of course, one can formulate a version of Lemma \ref{lemma_normalization} for any compact elliptic surface. But from now on, we will focus on elliptic-elliptic surfaces.

\begin{proposition}\label{prop_normalization-ellell}
    Let $S\to E$ be an elliptic-elliptic surface arising from quadratic base change of an elliptic K3 surface $X\to \mathbb P^1$ branched at three star fibers and one smooth fiber. Then the rational double cover $S\dashrightarrow X$ can be resolved via the following diagram:
    
\begin{figure}[ht]
    \centering
\begin{equation*}
\begin{tikzcd}
S'\arrow[r,"p"]\arrow[d,"\sigma"] & X\\
S,\arrow[ur,dashrightarrow,"\pi"']
\end{tikzcd}
\end{equation*}
\end{figure}

\noindent where $\sigma$ is the blow-up at $12$ points, and $p$ is branched double cover ramified along the exceptional divisors and one smooth fiber, which is $\sigma^*K_S$.
\end{proposition}
\begin{proof}
    Since the branch locus of $E\to \mathbb P^1$ contains the values under three star fibers, using Lemma \ref{lemma_normalization}, the normalization $S'$ of $E\times_{\mathbb P^1}X$ has $4\times 3=12$ curves with self-intersection $(-1)$. Contracting these curves, we obtain the relatively minimal model $S\to E$, which is an elliptic-elliptic surface (cf. Corollary \ref{cor_ell-ell-by-base-change}, (2)). The remaining smooth fiber where $\pi$ is ramified must be $K_S$ by the Riemann-Hurwitz formula.
\end{proof}

In particular, there is an involution $\iota'$ of $S'$ fixing 13 curve components, and $S'/\iota'\cong X$. By descending the involution to $S$ and commuting the $\iota$-quotient with the blow-up, we have:
\begin{corollary}\label{cor_12fixpoint}
    There is a symplectic involution $\iota: S\to S$, meaning that it acts trivially on $H^{2,0}(S)$, with $12$ isolated fixed points and a fixed curve $K_S$. The minimal resolution of $S/\iota$ is the K3 surface $X$.
\end{corollary}

This is formally analogous to the picture for Shioda--Inose structures for K3 surfaces; see \cite{vGS07}. The K3 surface $Y$ that arises from rational double cover of an elliptic Kummer K3 surface is branched along two star fibers, and each contains 4 components with odd multiplicity (cf. Lemma \ref{lemma_star-1comp}). Therefore, they correspond to the $8$ fixed points of the Nikulin involution on $Y$.

\subsection{Hodge theory of defective double cover}\label{sec_Hodge} We analyze the transcendental Hodge structure of the defective double cover of elliptic surfaces, and investigate a one-parameter family of elliptic-elliptic surface in Corollary \ref{cor_ell-ell-by-base-change}.
\begin{lemma}\label{lemma_TXpullback}In the setup of Proposition \ref{prop_normalization-ellell}, the maps  $\pi^*:T_X\hookrightarrow T_S$ and $\pi_*:T_S\hookrightarrow T_X$ are injective morphisms of $\Z$-Hodge structures. In particular,  there is an isomorphism of (unpolarized) $\bbQ$-Hodge structures $$T_X\otimes\bbQ\cong T_S\otimes \bbQ,$$
    and $\pi^*(T_X)$ is a finite index sublattice of $T_S$.
\end{lemma}


\begin{proof}
As explained in Section \ref{subsec_transcendental}, both $\pi_*$ and $\pi^*$ are well-defined on transcendental lattices for any rational map $\pi$.
Secondly,  $\pi$ has degree $2$, so $\pi_*\pi^*:T_X\to T_X$ is $2\cdot Id$. Hence, $\pi^*$ is injective. Moreover, by Hartogs theorem, the pullback of any holomorphic 2-form extends, so
$\pi^*\otimes\C:H^{2,0}(X)\to H^{2,0}(S)$ preserves the Hodge filtration and $\pi^*$ is a morphism of $\Z$-Hodge structures. Since $\dim H^{2,0}(S)=1$, $T_S\otimes \bbQ$ is an irreducible $\bbQ$-Hodge structure, and there is an isomorphism $\pi^*T_X\otimes \bbQ\cong T_S\otimes \bbQ$ of $\bbQ$-Hodge structures. In particular, $\pi^*T_X$ has finite index in $T_S$.
\end{proof}

\begin{remark}\normalfont \label{remark_Morrison}
  By \cite{Morrison84}, the morphisms $\pi_*: T_S(2)\to T_X$ and $\pi^*:T_X(2)\to T_S$ preserve the indicated lattice pairings.
\end{remark}

Next, we allow  the construction of Proposition \ref{prop_normalization-ellell} to vary in a family, and study the corresponding VHS.
\begin{lemma}\label{lemma_smoothfamily}
   There is a non-isotrivial family of elliptic-elliptic surfaces $f_U:\mathcal S_U \to U$ over a quasi-projective curve $U$, that arises from defective quadratic base changes of a fixed elliptic K3 surface $X\to \PP^1$ (cf. Corollary \ref{cor_ell-ell-by-base-change}) at three star fibers. More precisely, there is a non-isotrivial family $g_U:\mathcal E_U \to U$ of elliptic curves and map $h_U:\mathcal S_U \to \mathcal E_U$ such that $f_U = g_U\circ h_U$, and for every $t\in U$, the induced morphism on fibers
   $$h_t:S_t \to E_t$$
   is an elliptic-elliptic surface that is the Kodaira--N\'{e}ron model of the base change $E_t \times_{\PP^1} X$.
\end{lemma}
\begin{proof}
 For convenience, assume that the star fibers of $X\to \PP^1$ lie over the values 
 $$\{0,1,2\}\subset \PP^1.$$
 Let $C$ be a projective curve with a non-constant morphism $\lambda:C \to \PP^1$ of even degree, and let $\mathcal E$ be a cyclic double cover of $C\times \PP^1$ branched at the 2-divisible divisor
 $$D=\Gamma_\lambda + \sum_{j=0}^2 C\times \{j\},$$
 where $\Gamma_\lambda$ is the graph of $\lambda$.
 Let $U = C \smallsetminus \textup{pr}_1(D_{sing})$ be the open subset of $C$ over which $D$ is \'{e}tale. Observe that $g_U:\mathcal E_U \to U$ is an elliptic fibration with four distinguished sections, and it is non-isotrivial because $\lambda$ was non-constant. Now shrink $U$ further by removing $\lambda^{-1}(c)$ for any other critical values of $X \to \PP^1$ besides $\{0,1,2\}$.
 
 There is a morphism $\mathcal E_U \to \PP^1$ given by the composition of the double cover above with the second projection map $C\times \PP^1 \to \PP^1$. The fibered product
 $$\mathcal E_U \times_{\PP^1} X$$
 has normalization $\mathcal S'_U$, which is birational to the desired $\mathcal S_U$. 
We will construct a commutative diagram generalizing Proposition \ref{prop_normalization-ellell} to the case of a family over $U$.

    \begin{figure}[ht]
    \centering
\begin{equation*}
\begin{tikzcd}
\mathcal{S}'_U/U\arrow[r,"P"]\arrow[d,"\Sigma"] & X\times U\\
\mathcal{S}_U/U\arrow[ur,dashed,"\Pi"].
\end{tikzcd}
\end{equation*}
\end{figure}
Here, $P$ is a branched double cover, and $\Sigma$ will be a birational contraction.
Note that each fiber of $\mathcal S'_U$ contains 12 disjoint $(-1)$-curves. The monodromy action of $\pi_1(U)$ on this configuration is trivial, since they come from 12 curves in $X$, so each such exceptional curve traces out a divisor $R_i\subset \mathcal S_U$. The $P$-pullback of an ample divisor on $X\times U$ gives an ample divisor on $\mathcal S'_U$ which is also monodromy-invariant.
Hence, by the Relative Cone Theorem \cite[Theorem 3.25]{KollarMori98}, there is a sequence of contractions over $U$ of the ruled divisors $R_i$, whose composite $\Sigma:\cS_U'\to \cS_U$ results in the desired family $f_U:\mathcal S_U \to \mathcal E_U\to U$.

\end{proof}
\begin{proposition}\label{prop_trivialVHS}
The family constructed in Lemma \ref{lemma_smoothfamily} above satisfies the following.
 \begin{itemize}
       \item The VHS on $H^2$ is locally constant. More precisely, the polarized variation of Hodge structure $R^2f|{_U*}\Z$ admits a decomposition after tensoring with $\bbQ$ $$R^2f|{_U*} \mathbb{Q}= \mathcal T_{\bbQ}\oplus \mathcal{NS}_{\bbQ},$$ where $\mathcal T$ is the sublocal system defined by transcendental lattices $T_{S_t}$ and $ \mathcal{NS}=\mathcal T^{\perp}$ is the sublocal system defined by N\'eron--Severi groups $NS(S_t)$. The factor $\mathcal T_{\bbQ}$ has trivial monodromy,
      and  $\mathcal{NS}_{\bbQ}$ has finite monodromy.
\item The VHS on $H^1$ is non-constant.
   \end{itemize}
\end{proposition}
\begin{proof} The rational map 
$$\Pi: \mathcal S_U\dashrightarrow X\times U$$
gives a morphism of $\bbQ$-VHS
$$\Pi^*:H^2(X,\bbQ)\otimes \underline{\bbQ}_{U}\to R^2f|{_U*} \mathbb{Q},$$
whose stalks are $\pi_t^*: H^2(X,\bbQ)\to H^2(S_t,\bbQ)$. By Lemma \ref{lemma_TXpullback}, the image local system $\mathcal T_{\bbQ}=\textup{Im}(\Pi^*|_{T_X})$ is a constant sub-VHS isomorphic to $T_X\otimes \underline{\bbQ}_{U}$.


Since $NS(S_t)$ contains a monodromy-invariant ample class, the polarized N\'{e}ron--Severi is an integral PVHS of type $(1,1)$, so its monodromy is both compact and discrete, hence finite. In particular, $\mathcal{NS}$ is locally constant, and so is $R^2f|_{U*}\Z$. 

For the second statement, there is an isomorphism of Hodge structures $H^1(S_t,\Z)\cong H^1(E_t,\Z)$. Since $E_t$ varies non-isotrivially,  the last statement follows from Lemma \ref{lemma_smoothfamily}.

\end{proof}

\begin{corollary}\label{cor_constant-period}
    The weight 2 period map for the family $\mathcal S_U \to U$ is constant.
\end{corollary}

\begin{remark}\normalfont
   Note that by Corollary \ref{cor_constant-period}, infinitesimal Torelli fails for each surface $S_t$ in the family. This is consistent with a criterion in \cite{GZ} for infinitesimal Torelli in terms of the $j$-map at the canonical fiber $K_{S_t}$ associated with $S_t \to E_t$. 
\end{remark}

\subsection{Examples of families with constant VHS}
In this section, we revisit some examples of  one-dimensional families  of elliptic-elliptic surfaces  with constant period map \cite{Ikeda,EGW}. We construct them in a uniform way using Corollary \ref{cor_ell-ell-by-base-change}.  
\begin{example}\cite{Ikeda} \normalfont Ikeda found a 4-dimensional family $\mathcal I_1$ of elliptic-elliptic surfaces arising from Prym varieties of bielliptic curves of genus 3. The period image of $\mathcal I_1$ is 3-dimensional, and the generic surface in $\mathcal I_1$ has six $I_2$ fibers. 

These surfaces also arise from quadratic base change of an elliptic K3 surface $X\to \mathbb P^1$ with singular fibers of type $$I_2\times 3, I_0^*\times 3,$$  branched at the three $I_0^*$ fibers and an additional smooth fiber $X_t$. Note that a generic elliptic K3 surface with these singular fibers has $\rho=17$, and varies with 3 moduli. The moving branch point provides an additional parameter, so the moduli for the corresponding elliptic-elliptic surfaces with six $I_2$ fibers has dimension 4, and it dominates $\mathcal I_1$.
\end{example}

\begin{example}\cite{EGW} \normalfont Engel--Greer--Ward found a 6-dimensional family $\mathcal I_2$ of {\it isotrivial} elliptic-elliptic surfaces with six cuspidal fibers. The period image of $\mathcal I_2$ is  5-dimensional.

A subfamily of codimension two can also be constructed from quadratic base change of an elliptic K3 surface 
$X\to \mathbb P^1$
with singular fibers of type $$II\times 3, I_0^*\times 3.$$ Again, such elliptic K3 surfaces vary with 3 moduli, and the moving branch point provides an additional parameter. The family of elliptic-elliptic surfaces that arise from base change dominates a codimension two locus inside $\mathcal I_2$.
\end{example}

\section{Surface degenerations}\label{sec_degenerations}

Our construction begins with a specific instance of Corollary \ref{cor_ell-ell-by-base-change}. We start from  the Kummer surface $X=\Kum(E_{\omega}\times E_{\omega})$ associated to the self-product of the elliptic curve $E_{\omega}=\C/(\Z\oplus\omega\Z)$, where $\omega$ is a third root of unity. The K3 surface $X$ has Picard number $20$, and admits an elliptic fibration \cite{ShiIno77}
\begin{equation}\label{eqn_Kum(E0*E0)}
    X\to \mathbb P^1
\end{equation}
with three singular fibers of types $II^*,IV^*,I_0^*$. By choosing a suitable coordinate on $\mathbb P^1$, we may assume that the three singular fibers lie over $0,1,2$ respectively. The transcendental lattice of $X$ is given by 
  $$T_X\cong \mathrm A_2(2)=\begin{bmatrix}
        4&2\\2&4
    \end{bmatrix}.
$$
Let $\lambda:\PP^1\to \PP^1$ be a nonconstant morphism, and let $E_t\to \PP^1$ be the elliptic curve branched at $0,1,2$ and an additional point $\lambda(t)$.
Let $S_t\to E_t$ be the relatively minimal model of $X\times_{\PP^1} E_t$. Then by Lemma \ref{lemma_smoothfamily}, we have:
     \begin{lemma}\label{lemma_ell-ell-fibers}
        There exists a smooth family of elliptic-elliptic surfaces $$\{S_t\to E_t\}_{t\in U}$$ parameterized by an open subset $U\subset \mathbb P^1$, with 
        \begin{itemize}
            \item[(i)] constant VHS on the transcendental Hodge structure $T_{S_t}$;
            \item[(ii)] non-constant VHS on $H^1$;
            \item[(iii)] maximal Picard number $\rho=h^{1,1}=12$;
            \item[(iv)] the singular fibers of $S_t\to E_t$ are of type $IV^*$, $IV$;
            \item[(v)] the Mordell--Weil rank of $S_t/E_t$ is 2.
        \end{itemize}
     \end{lemma}
\begin{proof}
(i) and (ii) follow from Proposition \ref{prop_trivialVHS}. (iii) follows from Lemma \ref{lemma_TXpullback}; since $T_{S_t}$ has rank 2, $\rho=h^{1,1}=12$.  (iv) follows from Table \ref{table_star} showing that branched double covers of $II^*,IV^*,I_0^*$ become $IV^*,IV,I_0$, respectively. (v) follows from the Shioda--Tate formula \cite[Corollary 6.7]{SchShi19}; $\rho=2+r+6+2$, so the Mordell--Weil rank $r=2$. 
\end{proof}

\subsection{A degeneration of elliptic-elliptic surfaces}\label{subsec_projective-family}
In this subsection, we compactify the smooth family from Lemma \ref{lemma_ell-ell-fibers}. Our goal is to show the following

\begin{proposition}\label{prop_key}
   There exists a smooth projective 3-fold $\mathcal{S}$ and a proper flat family
  \begin{equation}\label{eqn_global-family}
    g:\cS\to \mathbb P^1   
    \end{equation}  which satisfies the following conditions:
    \begin{itemize}
        \item When $t\notin \{t_0,\ldots, t_5\}$, $S_t:=g^{-1}(t)$ recovers the family from Lemma \ref{lemma_ell-ell-fibers}.
        \item When $t=t_i$, the fiber $g^{-1}(t)$ is simple normal crossings (SNC), with a component that is reduced and birational to a K3 surface.
    \end{itemize}
\end{proposition}

We begin by compactifying the base family of elliptic curves to an elliptic surface $\E$. There are many ways of doing this, using the general setup of Lemma \ref{lemma_smoothfamily}, but for the sake of concreteness we give the simplest example here. Our elliptic surface $\mathcal E$ will be the minimal resolution of a certain double cover of $\PP^1\times \PP^1$. In order to distinguish the two factors, we will denote the first by $\PP^1_t$ and the second by $\PP^1_x$.
\begin{lemma}\label{lemma_RES}
     There is a rational elliptic surface $f:\E\to \mathbb P^1_t$ with six $I_2$ fibers that admits a map $h:\E\to \mathbb P_x^1$, such that the restriction to the general fiber of $f$ 
     is a double cover $h|_{E_t}:E_t\to \mathbb P^1_x$ branched at four points, including three fixed points $x=0,1,2$.
 \end{lemma}

\begin{proof}
The elliptic surface $\E$ is the minimal resolution of the double cover of $\PP_t^1\times \PP^1_x$ branched along a divisor $D$ of bidegree $(4,2)$, pictured below. The divisor $D$ is the sum of 3 horizontal lines with the graph of a general double cover $\lambda:\mathbb P^1_t\to\mathbb P^1_x$. The first projection $f:\E\to \mathbb P^1_t$ defines an elliptic fibration surface with fiber $E_t= f^{-1}(t)$. The second projection is a conic bundle. When $t$ is general, under the projection $\E\to \PP^1_x$, $E_t$ can be realized as a double cover of $\PP^1_x$ branched at $\{0,1,2,\lambda(t)\}$.
The Kummer K3 surface $X$ has a special elliptic fibration $X \to \PP^1_x$ with three star fibers situated over $\{0,1,2\}\in \PP^1_x$, as indicated in the figure.



\begin{figure}[h]
\centering
\includegraphics[width=0.6\textwidth]{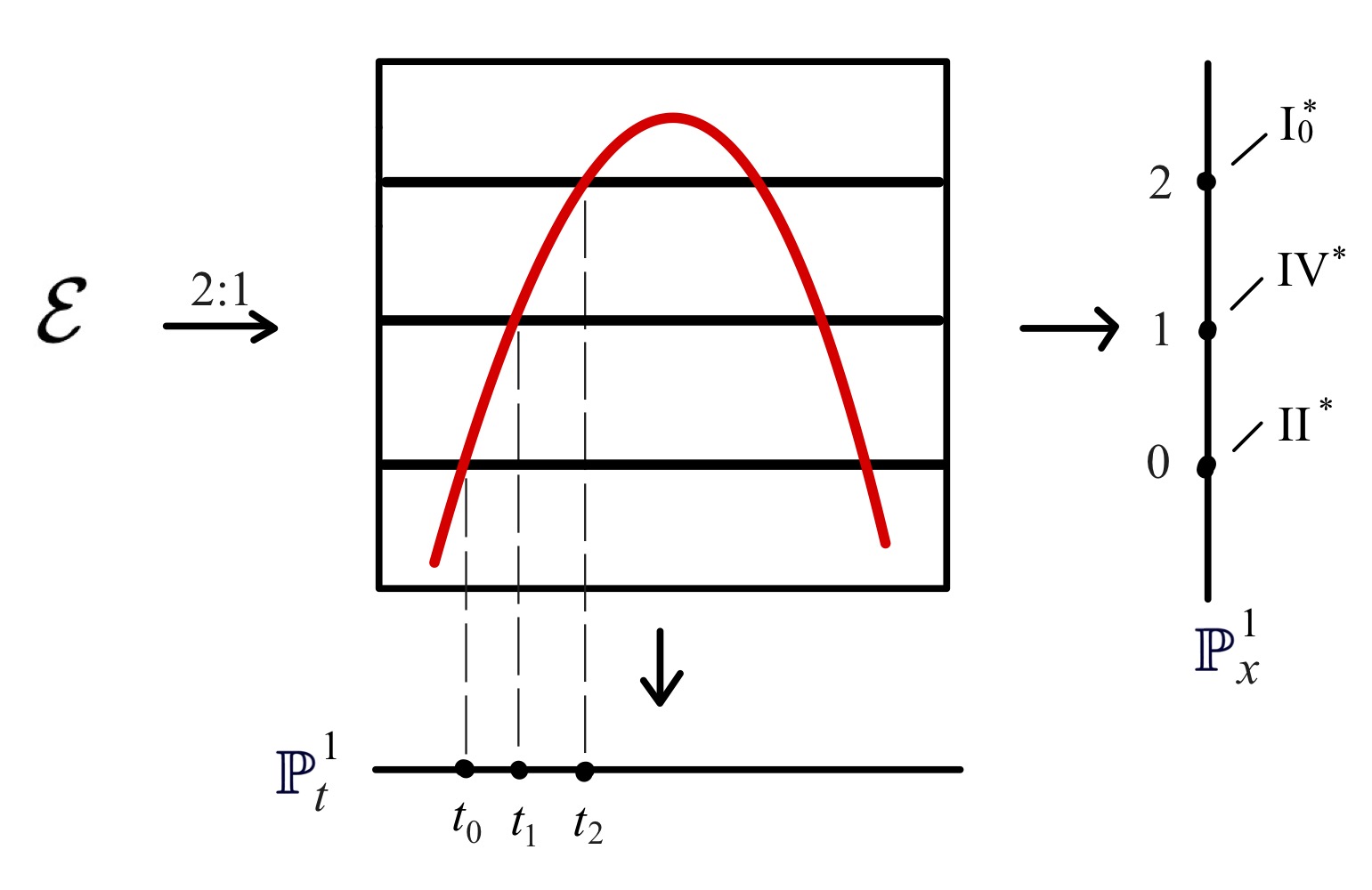}
\caption{Rational elliptic surface with its two projections}
\label{figure_24curve}
\end{figure}

 By choosing $\lambda$ general, the six singular points of $D$ are nodes. Accordingly, the branched double cover will have an $A_1$-singularity over each singular point of $D$, so its desingularization $\mathcal E$ gives an elliptic surface $\E\to \mathbb P^1_t$ with six $I_2$ fibers. Each $I_2$ fiber consists of a main component, the proper transform of the nodal double cover of $\PP^1_x$, together with an exceptional component from the blow up of the $A_1$-singularity. On the main component, $h$ gives a double cover $\PP^1 \to \PP^1_x$ branched at the two values in $\{0,1,2\}$ not equal to $\lambda(t_i)$.
\end{proof}
These singular fibers of $f:\mathcal E \to \PP^1_t$ come in pairs according to the value of $\lambda(t)\in \mathbb P^1_x$. 
We label these values of $t$ as follows:
\begin{align*}
\lambda^{-1}(0)&=\{t_0, t_3\},\\
\lambda^{-1}(1)&=\{t_1, t_4\},\\
\lambda^{-1}(2)&=\{t_2, t_5\}.
\end{align*}
The fibered product $\mathcal E\times_{\PP^1_x} X$ has normalization $\mathcal S'$ admitting a proper morphism,
$$\mathcal S' \to \PP^1_t,$$
This is a family of elliptic-elliptic surfaces which is smooth over $U=\PP^1\smallsetminus \{t_0,t_1,t_2,t_3,t_4,t_5\}$. We showed in Lemma \ref{lemma_smoothfamily} that $\mathcal S'_U$ admits a contraction $\mathcal S'_U \to \mathcal S_U$ which restricts to the minimal model contraction on each smooth surface fiber.

Let $\overline{\mathcal S}\to \PP^1$ denote any flat projective family of surfaces extending $\mathcal S_U$, and let $\mathcal S\to \overline{\mathcal S}$ be a resolution of singularities of the total space such that the fibers have simple normal crossings. We can also arrange that the fibers over each pair $t_0, t_3$ etc. are isomorphic.



\begin{lemma}\label{lemma:red}
    Each singular fiber of $\mathcal S' \to \PP^1$ and of $g:\mathcal S\to \PP^1$ contains a unique reduced component birational to a K3 surface.
\end{lemma}
 
\begin{proof}  The singular fibers of $\mathcal S' \to \PP^1$ over $\{t_i\}_{i=0,\dots,5}$ are readily described. First, observe that the fibers of $\mathcal E \to \PP^1$ are $I_2$ configurations, featuring two reduced components isomorphic to $\PP^1$ meeting at two nodes. The fiber of
$$\mathcal E \times_{\PP^1_x} X \to \PP^1_t$$
over $t_i$ is the union of $Z_{1,i}$ and $Z_{2,i}$. The (main) component $Z_{1,i}$ is reduced and birational to the base change of $X \to \PP^1_x$ branched at the two (other) star fibers. On the other hand, $Z_{2,i}$ is described as follows:
    $$Z_{2,0}\cong Z_{2,3}\cong II^*\times \mathbb P^1,\ Z_{2,1}\cong Z_{2,4}\cong IV^*\times \mathbb P^1,\ Z_{2,2}\cong Z_{2,5}\cong I_0^*\times \mathbb P^1.$$
Each $Z_{2,i}$ contains nonreduced components due to the nonreduced curves in the star fibers.
After normalization, the main component $Z_{1,i}^\nu$ is a K3 surface blown up at 8 points, which proves the first statement. For the second statement, recall from Lemma \ref{lemma_smoothfamily} that we have a {\it birational} map
$\mathcal S' \dasharrow \mathcal S$
with indeterminacy locus supported in the central fiber. Resolving this indeterminacy to produce a regular map,
$$\xymatrix{
\widehat{\mathcal S'} \ar[r]^\eta \ar[d] & \mathcal S \\
\mathcal S' \ar@{..>}[ur] & 
}$$
the blow up $\widehat{\mathcal S'}$ still contains a component birational to a K3 in each singular fiber. These components are non-uniruled divisors, and $\mathcal S$ is smooth by assumption, so no such component is contracted by $\eta$, using \cite[Prop. 1.3]{KollarMori98}.
\end{proof}

\subsection{The K3 components}\label{Section_SIcousins}
In this subsection, we will describe the geometry of the three K3 surfaces whose birational models occur in Proposition \ref{prop_key}.

\begin{notation}\normalfont
    Let $Y_0$ (respectively, $Y_1$, and $Y_2$) be the K3 surfaces from Proposition \ref{prop_key}. They are the minimal models of the main components of the fiber over $t_0,t_1,t_2$ of $\mathcal S \to \PP^1$. (Note the singular fibers over $t_3,t_4,t_5$ are identical to $t_0,t_1,t_2$, respectively.)
\end{notation}
Each $Y_i$ is obtained by taking the relatively minimal regular model of the normalized base change $\mathbb P^1\times_{\mathbb P^1_x}X$, where $\mathbb P^1\to \mathbb P^1_x$ is the double cover branched at $x=1,2$ (respectively, $x=0,2$, and $x=0,1$). 




\subsubsection{Analysis of $Y_0$: }
By construction, the elliptic K3
$$Y_0\to \mathbb P^1$$
arises from a quadratic base change of \eqref{eqn_Kum(E0*E0)} branched at $IV^*$ and $I_0^*$. Hence by Table 2, it has singular fibers of types $II^*, II^*, IV$. Lemma \ref{lemma_ShiodaInose} implies that $T_{Y_0}$ has intersection pairing 
$T_{X}=A_2$, which  has discriminant $3$.
Theorem \ref{thm_singularK3} guarantees that $Y_0$ is the K3 surface arising from the minimal desingularization of $(E_\omega\times E_\omega)/\mu_3$, acting anti-diagonally. This is one of Vinberg's ``most algebraic'' K3 surfaces \cite{Vin83}.

\subsubsection{Analysis of $Y_1$:} 
The elliptic K3 
$$Y_1\to \mathbb P^1$$
arises from a quadratic base change of \eqref{eqn_Kum(E0*E0)} branched at $II^*,I_0^*$ fibers, hence it has three singular fibers of type $IV^*$. The Shioda--Tate formula implies the elliptic surface is extremal. By the classification theorem \cite[Table 2, No. 219]{Shimada-Zhang01}, the Mordell--Weil group is $MW(Y_1/\mathbb P^1)\cong \Z/3\Z$. Moreover,  the transcendental lattice of $Y_1$ is isometric to  $T_{Y_0}$. In particular

\begin{proposition}\label{prop_disc3}
There is an isomorphism $Y_0\cong Y_1$, and $\disc(Y_0)=\disc(Y_1)=3$.
\end{proposition}

\subsubsection{Analysis of $Y_2$:}
The elliptic K3
\begin{equation}\label{eqn_Y2->P^1}
    Y_2\to \mathbb P^1
\end{equation}
arises from base change of \eqref{eqn_Kum(E0*E0)} branched at $II^*$ and $IV^*$ fibers. Hence it has singular fibers $I_0^*,I_0^*, IV,IV^*$. Our claim is that


\begin{proposition}\label{prop_Y3isnotY2Y1}
    $Y_2\ncong Y_1\cong Y_0$. 
\end{proposition}
\begin{proof}
   This follows from the fact that Vinberg's most algebraic K3 surface $Y_0$ does not have an elliptic fibration structure with singular fiber types
   $I_0^*,I_0^*, IV,IV^*$ (cf. \cite{Vin83, Nis96, BZ25}). 
\end{proof}

We have a more precise description of the K3 surface $Y_2$:

\begin{proposition}\label{prop_disc(T_2)}The transcendental lattice 
$T_{Y_2}$ has discriminant 
$$\disc(Y_2)=48.$$ Moreover, the elliptic fibration \eqref{eqn_Y2->P^1} has Mordell--Weil group $\Z^2$ and the Mordell--Weil lattice has discriminant 
$$\disc(\textup{MWL})=\frac13.$$
\end{proposition}
\begin{proof}
    Let $\pi:Y_2\dashrightarrow X$ be the rational double cover, then 
    \begin{itemize}
        \item $\pi_*\pi^*=[2]|_{T_X}$,
        \item $\pi^*\pi_*=[2]|_{T_{Y_2}}$,
    \end{itemize}
    where $[2]$ denotes multiplication by $2$.
    Hence both $\pi_*$ and $\pi^*$ are injective. Since the image of $[2]|_{T_X}$ has index $2^{rk(T_X)}=2^2$, $\pi_*T_{Y_2}$ has index $2^{\alpha}$ in $T_X$, with the possible values $\alpha=0,1,2$. Recall that $T_X=\mathrm A_2(2)$ has discriminant $3\cdot 2^2$,

$$(2^{\alpha})^2=\frac{\disc(\pi_*T_{Y_2})}{\disc(T_X)}=\frac{\disc(\pi_*T_{Y_2})}{3\cdot 2^2}.$$
We have $\disc(\pi_*T_{Y_2})=3\cdot 2^{2+2\alpha}$.
On the other hand, the pushforward $\pi_*$ scales the intersection pairing on transcendental lattices by $2$. Hence, $\disc(\pi_*T_{Y_2})=2^2\cdot \disc(T_{Y_2}).$
Combining these identities, we obtain $$\disc(T_{Y_2})=3\cdot 2^{2\alpha},\ \alpha=0,1,2.$$

However, if $\alpha= 0$, then $\disc(T_{Y_2})=3$, which contradicts Proposition \ref{prop_Y3isnotY2Y1}; if $\alpha= 1$, then $\disc(T_{Y_2})=12$, and there are two possibilities up to isometry: 
$$(1)\ T_{Y_2}= \begin{bmatrix}
        4&2\\2&4
    \end{bmatrix},\ \textup{or}\ (2)\ T_{Y_2}= \begin{bmatrix}
        2&0\\0&6
    \end{bmatrix}.
$$ Note that the first one is the Kummer K3 surface $X=\Kum(E_\omega\times E_\omega)$ that we started with. However,
there is a complete classification of elliptic fibrations for K3 surfaces of types (1) and (2) (see \cite{Nis96} and \cite{BGH+}), but none of them have the fibrations with singular fibers $I_0^*\times 2$,$IV^*$, $IV$. Therefore $\alpha\neq1$, and the only possibility is $\alpha=2$, hence $\disc(T_{Y_2})=48$.

We already know $Y_2$ has maximal Picard number $\rho=20$. By Shioda--Tate formula, $Y_2\to \mathbb P^1$ has Mordell--Weil rank $2$. The configurations of the singular fibers implies that $MW_{\mathrm{tor}}=0$ \cite[Proposition 6.33 (iv)]{SchShi19}. Therefore, the Mordell--Weil group is $\Z^2$. The discriminant formula \cite[Corollary 6.39]{SchShi19} implies
$$\disc(NS(Y_2))=\disc(T_{Y_2})=3\times 3\times 4^2\times \disc(\textup{MWL}).$$
Hence $\disc(\textup{MWL})=\frac13$. 
\end{proof}
\begin{remark}[Mordell--Weil lattice]\normalfont
    A second proof that $Y_2$ is not Vinberg's most algebraic K3 surface (Proposition \ref{prop_Y3isnotY2Y1}) independent from the classification of elliptic fibration structures goes as follows. Assuming the contrary, $\disc(T_{Y_2})=3$. The discriminant formula for Mordell--Weil lattices implies that $\det(\textup{MWL})=\frac{1}{48}$. However, from the formula of the height pairing \cite[Theorem 6.24]{SchShi19} and knowledge of the singular fibers $(I_0^*,I_0^*, IV,IV^*)$, the lcm of the denominator of each entry in the Mordell--Weil lattice is $6$, which implies that the denominator of $\disc(\textup{MWL})$ divides $6\times 6=36$. This is a contradiction.
\end{remark}

\section{Failure of the local invariant cycle theorem over $\Z$}\label{sec_LICTZ}

The goal of this section is to prove Theorem \ref{main_thm}. It will follow from the following theorem.

\begin{theorem}\label{main_thm-equiv}
For any flat projective completion of the family $\mathcal{S}_U$ with smooth total space, the \textup{LICT}$_{\Z}$ fails near $t_2$. It fails also for any semistable reduction near $t_2$.
\end{theorem}

\subsection{The transcendental lattice of a singular surface.}
We now define the transcendental lattice of a reducible surface, extending the discussion in Section \ref{subsec_transcendental}.

 Given a singular (possibly reducible) projective surface $Z$, recall that there is a mixed Hodge structure on $H^2(Z,\bbQ)$ with an increasing weight filtration $W_kH^2(Z,\bbQ)$ and decreasing Hodge filtration $F^pH^2(Z,\bbQ)$  (cf. \cite{Deligne-HodgeII}). In particular, $Gr_2^W H^2(Z, \bbQ)$ is a pure Hodge structure of weight two. Working integrally, we have
$$Gr^W_2H^2(Z,\Z) := \frac{H^2(Z, \Z)}{W_1H^2(Z,\bbQ)\cap H^2(Z, \Z)}.$$
 The following definition generalizes Lemma \ref{lemma_trans-irreducible}.
\begin{definition}\label{def_sing-weight-Z}\normalfont
Define
$$T_{Z,\bbQ} \subset Gr_2^W H^2(Z, \bbQ)$$
to be the smallest $\bbQ$ sub-Hodge structure whose complexification contains the $(2,0)$ part on the right. Define
\begin{equation*}
    T_Z := T_{Z,\bbQ} \cap  Gr^W_2H^2(Z,\Z)/\textup{torsion}.
\end{equation*}

\end{definition}

\begin{lemma}\label{lemma:H2trSNC}
     Suppose that $Z= \bigcup_{i=1}^N Z_i$ is an SNC surface. Then there is an isomorphism of $\Z$-Hodge structures
     \begin{equation}\label{eqn_TZ}
         T_{Z}\cong \oplus_iT_{Z_i}.
     \end{equation}
\end{lemma}

\begin{proof}
We can compute the cohomology of $Z$ by a Mayer--Vietoris spectral sequence. With rational coefficients, it  degenerates at the second page $E_2$, and computes the weight graded subquotients of the mixed Hodge structure.
In particular, we obtain an exact sequence of Hodge structures
\begin{equation}\label{eq:GrH2Z}
0\to Gr^W_2H^2(Z,\Z)/\text{torsion}\to \oplus_i H^2(Z_i,\Z)/\text{torsion}\xrightarrow{\rho} \oplus _{i<j}H^2(Z_{ij},\Z),
\end{equation}
where $Z_{ij}= Z_i\cap Z_j$. Since $\rho$ is the restriction map $\alpha_i\mapsto \alpha_i|_{Z_{ij}}=\alpha_i\cdot [Z_{ij}]$, then if $\alpha_i\in T_{Z_i}$, $\alpha_i\in \ker(\rho)$. Therefore
$$\oplus_i T_{Z_i}\subset Gr_2^WH^2(Z,\Z)/\textup{torsion}.$$


By definition, $T_Z$ is the primitive closure of $\oplus_iT_{Z_i}$ inside $Gr^W_2H^2(Z,\Z)/\textup{torsion}$. But each summand $T_{Z_i}$ is already primitive in $H^2(Z_i,\Z)/\textup{torsion}$, therefore we obtain
$$T_Z \cong \oplus_i  T_{Z_i}.$$
\end{proof}

\subsection{Proof of the main theorem}
In this subsection we prove Theorem \ref{main_thm-equiv}. In particular, we show that a semistable model near one of the points $p\in \mathbb P^1\smallsetminus U$ of the family \eqref{eqn_global-family} gives a counterexample to the LICT$_{\Z}$.

Let $U_i$ be an analytic neighborhood of $t_i$, isomorphic to a disk, and set $U_i^* = U_i\smallsetminus\{t_i\}$. Assume $U_0^*\cap U_1^*\cap U_2^*\neq\emptyset$ and choose a point $t$ in the intersection. Let
\begin{equation}\label{eqn_anafamily}
    g_i:\mathcal{S}_i\to U_i,\  i=0,1,2
\end{equation}
be the restriction of the family $\cS\to \mathbb P^1$ (cf. Proposition \ref{prop_key}) to $U_i$. Write $S_{i,0}:=g_i^{-1}(t_i)$ and $S_t:=g_i^{-1}(t)$. Recall that the fiber $S_{i,0}$ is an SNC surface and contains a reduced component $Y_i'$ birational to the K3 surface $Y_i$ described in the previous section.

\begin{lemma}\label{lemma_wrongway}
   The restriction map $H^2(S_{i,0},\Z)\to H^2(Y_i',\Z)$ induces an isomorphism of transcendental Hodge structures
   \begin{equation}\label{eqn_rest_T}
       T_{S_{i,0}}\xrightarrow{\sim} T_{Y_i'}.
   \end{equation}
\end{lemma}
\begin{proof}
Let $m$ denote the least common multiple of the multiplicities of the components of ${S}_{i,0}$.
Let $\breve{U}_i\to U_i$ denote the degree $m$  cyclic cover ramified at $t_i$,  let $\breve{\mathcal{S}}_i$ denote the normalization
of the fibered product $\mathcal{S}_i\times_{U_i} \breve{U}_i$, and $\breve{g}_i: \breve{\mathcal{S}}_i\to \breve{U}_i$ the projection.
Then by \cite[Lemma 2.2]{steenbrink2} $\breve{S}_{i,0}:={\breve{g}_i}^{-1}(t_i)$ is reduced, and each component has rational singularities. Since
we have  a holomorphic surjection from $\breve{S}_{i,0}$ to ${S}_{i,0}$, it follows that each component
of $S_{i,0, red}$ is dominated by a component of $\breve{S}_{i,0}$.
By \cite[theorem 2.11]{steenbrink2}, 
 $R^2\breve{g}_{i*}\OO_{\breve{\cS}_i}$ is locally free.
 Since the nearby smooth surface has $h^{0,2}(\breve{S}_t)=p_g(\breve{S}_t)=1$ (cf. Lemma \ref{lemma_ell-ell-fibers}), it follows that exactly one component of $\breve{S}_{i,0}$ has
 $ p_g=1$ and the  other components  have $p_g=0$ and trivial transcendental Hodge structures. Therefore, the same conclusion holds for $S_{i,0, red}$.
 Then the result follows from Lemma \ref{lemma:H2trSNC}.
\end{proof}

We have a specialization map $sp_i^{k}$ given as a composition
$$sp_i^{k}: H^k(S_{i,0},\Z)\cong H^k(\mathcal{S}_i,\Z)\to H^k(S_{t},\Z),$$
where the first map is induced from a deformation retract $\mathcal{S}_i\to S_{i,0}$ onto the special fiber, and the second map is the restriction to a general fiber. $sp_i^{k}$ is a morphism of mixed Hodge structures.
\begin{remark}\normalfont
   Since the monodromy is finite when $k=2$ (see Proposition \ref{prop_trivialVHS}), the limiting mixed Hodge structure is pure. One can define the limiting mixed Hodge structure $H^2_{\lim}(\cS_i,\bbQ)$ with quasi-unipotent monodromy using the definition \cite{KerrLaza}. It can also be done using semistable reduction \cite{KKFM-ssreduction}: Up to taking a finite base change $\tilde{U}_i\to U_i$ and birational modification of $\cS_i\times_{U_i}\tilde{U}_i$, there is a semistable family $\tilde{\mathcal S}_{i}\to \tilde{U_i}$ with trivial monodromy, and there is an isomorphism of pure Hodge structures of weight two
  $$H^2(S_{t}, \bbQ) \cong H^2_{\lim}(\tilde{\cS}_{i}, \bbQ).$$
In particular, there is a commutative diagram with exact rows:
 \begin{figure}[ht]
    \centering
\begin{equation*}
\begin{tikzcd}
H^2(\tilde{S}_{i,0},\bbQ)\arrow[r]& H^2_{\lim}(\tilde \cS_i,\bbQ)\arrow[r]&0\\
H^2(S_{i,0},\bbQ)\arrow[u]\arrow[r]&H^2_{\lim}(\cS_i,\bbQ)^{\pi_1(U_i,t)}\arrow[u]\arrow[r]&0.
\end{tikzcd}
\end{equation*}
\end{figure}

 We note that by our construction, the local monodromy action of $\pi_1(U_i^*,t)$ on the transcendental $T_{S_t}$ is trivial (see Proposition \ref{prop_trivialVHS}), the restriction of the diagram above to the transcendental part induces an isomorphism of $\bbQ$-Hodge structures:
 $$T_{S_{i,0},\bbQ}\cong T_{\tilde{S}_{i,0},\bbQ}\cong T_{S_t,\bbQ}.$$
\end{remark}
In particular, we have the following:
\begin{lemma}
 The specialization map $sp_i^{2}$ induces an injection 
  \begin{equation}\label{eqn_inj}
T_{S_{i,0}}\xhookrightarrow{sp_i^2}T_{S_t}  
  \end{equation}
between the transcendental Hodge structures and the image has finite index. 
\end{lemma}

Now we compare the lattice structures on the two ends of \eqref{eqn_rest_T} and \eqref{eqn_inj}.

\begin{lemma}
    The homomorphism  given by the composition
    $$ T_{Y_i'}\xleftarrow{\sim}T_{S_{i,0}}\xrightarrow{sp_i^2} T_{S_t}$$
    respects intersection pairings.
\end{lemma}

\begin{proof}
Let $S_{i,0}= Y_i' \cup Z$ where $Y_i'$ is the unique component with $p_g=1$ and  $Z= \bigcup Z_j$ is the union of the remaining components, which
have $p_g=0$.  It is not difficult to check that
the composition
$$ 
H^4(Y_i',\Z) \oplus \bigoplus_j H^4(Z_j,\Z) \xleftarrow{\sim}  H^4(S_{i,0}, \Z) \to H^4(S_t, \Z) 
$$
of the inverse of the first map with the second yields an isomorphism $ H^4(Y_i',\Z)\cong H^4(S_t, \Z) =\Z$,
when restricted to the first component. 
By Lemma \ref{lemma_wrongway}, $H^2_{tr}(S_{i,0},\Z)= H^2_{tr}(Y_i', \Z)$. To be more
explicit,  an element $\alpha$ of $ H^2_{tr}(S_{i,0},\Z)$ can be realized as a vector $$(\alpha_0,0,\ldots,0)\in H_{tr}^2(Y_i',\Z) \oplus \bigoplus_j H^2(Z_j,\Z),$$ where all but first component is zero. 
Since the Mayer-Vietoris spectral sequence is multiplicative,
the product of two such vectors corresponding 
to $\alpha,\beta\in H_{tr}^2(S_{i, 0},\Z)$ would  be
$(\alpha_0\cup\beta_0,0,\ldots,0 )\in H^4(Y_i',\Z)\oplus \bigoplus_j H^4(Z_j, \Z)$.
Using this 
and the fact that
$sp^*_i$ is ring homomorphism, for $\alpha,\beta\in T_{S_{i, 0}}$, we have
$$\int_{S_{t}}sp_i^2( \alpha)\cup sp_i^2(\beta) = \int_{S_{i,0}} \alpha\cup \beta= \int_{Y_i'} \alpha_0\cup \beta_0.$$
\end{proof}

\begin{lemma}\label{lemma:sp1}
    For $i=0$ or $1$, the map \eqref{eqn_inj} is an isomorphism. In particular,
    the transcendental lattice of $S_t$ is isometric to $$ T_{S_t}\cong\mathrm A_2,$$
    which has discriminant $\disc (\mathrm A_2)=3$. 
    For $i=2$, the image is an index $4$ subgroup of $T_{S_t}$.
\end{lemma}

\begin{proof}
When $i=0,1$, then 
by Proposition \ref{prop_disc3} and Lemma \ref{lemma_wrongway} the lattice 
$T_{S_{i,0}}\cong \mathrm A_2$. Since $\disc(\mathrm A_2)=3$ is squarefree, $\mathrm A_2$ does not admit any proper integral overlattice. The injectivity of \eqref{eqn_inj} forces the equality
$$\textup{Im}(sp_i^2)=T_{S_{t}}.$$
When $i=2$, the discriminant of the lattice $T_{S_{i,0}}$ is $48$ by Proposition \ref{prop_disc(T_2)}. Hence $\textup{Im}(sp_2^2)$ must be a proper sublattice of index 
$\sqrt{48/3}=4$.
\end{proof}

Finally, Theorem \ref{main_thm-equiv} will follow from:

\begin{theorem}\label{main_thm-aux}
The map $sp_2^2$ is not surjective. Therefore, \textup{LICT}$_{\Z}$ fails for the family $$\mathcal{S}_2\to U_2,$$
and likewise for any semistable reduction $\tilde{\mathcal S}_{2}\to \tilde{U}_2$.
\end{theorem}
\begin{proof}
If  $sp_2^2$ were surjective, then its restriction to the transcendental part would be surjective. However, this contradicts Lemma \ref{lemma:sp1}. For the second half of the argument, note that the K3 component in the original special fiber is reduced by Lemma \ref{lemma:red}, so its birational type and therefore the transcendental lattice remains unchanged under any further base change and birational modification on the total space. Hence, the same argument applies.
\end{proof}



\section{Failure of the global invariant cycle theorem over $\Z$}\label{sec_GICTZ}
In this section, we prove Theorem \ref{intro_thm-failureGICTZ}. First, recall the global invariant cycle theorem over $\bbQ$. Also see \cite{Voisin-volII} for a clear exposition.

\begin{theorem}[GICT$_{\bbQ}$ \cite{Deligne-HodgeII}]\label{intro_thm_GICTQ}
    Let $f:\mathcal{X}\to B$ be a proper surjective  morphism between smooth projective varieties. Let $U\subseteq B$ be the largest Zariski open subset over which $f$ is smooth. Let $X_t$ be a general fiber. Then the image of the restriction map
$$H^n(\mathcal{X},\bbQ)\to H^n(X_t,\bbQ)$$
    is the subspace of invariants 
    $ H^n(X_t,\bbQ)^{\pi_1(U,t)}$.
\end{theorem}

For the integral version of global invariant cycle theorem (\textup{GICT}$_{\Z}$), we have the following observation.

\begin{lemma}\label{lemma_failureGICTZ}
    Let $f:\mathcal{X}\to B$ be a family with the same assumptions as in Theorem \ref{intro_thm_GICTQ}. If the following two conditions hold:
    \begin{enumerate}
        \item  There is a constant sublocal system $\mathcal{T}$ of $R^nf|_{U*}\Z$;
        \item The \textup{LICT}$_{\Z}$ fails for a class in $\mathcal T$ around a singular fiber over $p\in B$. 
    \end{enumerate}
Then the \textup{GICT}$_{\Z}$ fails for $f$ in degree $n$, i.e., the restriction map 
$$H^n(\mathcal X,\Z)\to H^n(X_t,\Z)^{\pi_1(U,t)}$$
is not surjective.
\end{lemma}
\begin{proof}
Let $\Delta_p$ be an analytic neighborhood of $p\in B$. Then the restriction map from cohomology on the global total space $\mathcal{X}$ to the fiber $X_t$ factors through the cohomology on the local analytic family $\mathcal{X}|_{\Delta_p}=f^{-1}(\Delta_p)\to \Delta_p$: 
$$r:H^n(\mathcal{X},\Z)\to H^n(\mathcal{X}|_{\Delta_p},\Z)\xrightarrow{r_p} H^n(X_t,\Z).$$
If $r_p$ is not surjective for a globally invariant class $\alpha_t \in \mathcal T|_t$, then neither is $r$.
Hence, GICT$_{\Z}$ fails for $f$ in degree $n$.
\end{proof}

\begin{theorem}\label{thm_GICTZ_aux}
    The global invariant cycle theorem \textup{GICT}$_{\Z}$ fails over $\Z$ for the family $\cS\to \mathbb P^1$ from Proposition \ref{prop_key}. It also fails for any base change or birational modification, and in particular for any semistable model of the family.
\end{theorem}
\begin{proof}
By construction, the family has globally trivial monodromy on transcendental cycles (cf. Proposition \ref{prop_trivialVHS}), which forms the constant sublocal system $\mathcal T$ of $R^2g|_{U*}\Z$. Then the claim follows from Theorem \ref{main_thm-aux} and Lemma \ref{lemma_failureGICTZ}.
\end{proof}

\section{Another counterexample}\label{sec_addcounterexample}





In this section, we construct a second counterexample to both the local and global ICT$_{\Z}$, using a similar family of elliptic-elliptic surfaces, only this time it is associated to Vinberg's {\it second} most algebraic K3 surface. This is the minimal resolution of the $(E_i\times E_i)/\mu_4$, acting anti-diagonally, where $E_i= \C/(\Z+\Z i)$. This  has discriminant $4$ \cite{Vin83}. 

We start from $X=\Kum(E_i\times E_i)$, the Kummer surface associated with the self-product of the elliptic curve $E_i$.  It has transcendental lattice $T_X\cong \mathrm A_1(2)\oplus \mathrm A_1(2)
$ and admits an elliptic fibration \cite[Table 2, \#313]{Shimada-Zhang01}  
\begin{equation}\label{eqn_Kum(E_i*E_i)}
    X\to \mathbb P^1 
\end{equation}
with singular fibers of type $II^*, I_1^*\times 2$ over $0,1,2$ respectively. 
The double cover branched along two $I_1^*$ fibers produces an elliptic K3 surface $Y_0$ with singular fibers 
$II^*\times 2, I_2\times 2,$
and by Lemma \ref{lemma_ShiodaInose}, it has transcendental lattice $$T_{Y_0}\cong \mathrm A_1\oplus \mathrm A_1=\begin{bmatrix}
        2&0\\0&2
    \end{bmatrix}.
$$
In particular, by Theorem \ref{thm_singularK3}, $Y_0$ is Vinberg's second most algebraic K3 \cite{Vin83}, \cite[Table 2, \#296]{Shimada-Zhang01}.

The other two double covers of $X$, branched along $II^*$ and one of $I_1^*$ fibers, produce the same elliptic K3 surface $Y_1\cong Y_2$ with singular fibers 
$IV^*, I_2,I_1^*\times 2.$
However, these fiber types do not appear in the classification of elliptic fibration structures on $Y_0$ \cite{Nis96}. Hence $Y_0\ncong Y_1\cong Y_2$. Using Corollary \ref{cor_ell-ell-by-base-change}, double covers of \eqref{eqn_Kum(E_i*E_i)} branched at the three critical values and another moving point produce a family of elliptic-elliptic surfaces $S_t\to E_t$, each with singular fibers of type $IV^*,I_2\times 2$ and Mordell--Weil rank $2$. A construction completely analogous to the previous sections shows that
\begin{theorem}\label{thm_2nd-counterexample}
   There is a smooth projective fibered 3-fold  $$g:\cS\to \mathbb P^1,$$ 
whose general fiber is a smooth elliptic-elliptic surface $S_t$ with transcendental lattice  
$T_{S_t}\cong \mathrm A_1\oplus \mathrm A_1$. The singular fibers of $g$ are SNC surfaces containing a reduced component $Y_i'$ birational to one of the K3 surfaces $Y_i$, with $i=0,1,2$. Moreover, 
    \begin{enumerate}
        \item  \textup{LICT}$_{\Z}$ fails for the analytic family around the singular fiber containing $Y_i'$ with $i=1,2$, as well as for any local semistable reduction; 
        \item  \textup{GICT}$_{\Z}$ fails for $g$, as well as for any global simultaneous semistable reduction.
    \end{enumerate}
\end{theorem}
\noindent\textit{Sketch of Proof.}
    The only essential difference for the proof is that the transcendental lattice of the nearby smooth surface $S_t$ is isometric to
    \begin{equation}\label{eqn_2ndVinberg-lattice}
       T_{S_t}\cong T_{Y_0}\cong \mathrm A_1\oplus \mathrm A_1.
    \end{equation}
    This is because there is an inclusion of lattices 
    $i:T_{Y_0}\cong \mathrm A_1\oplus \mathrm A_1\hookrightarrow T_{S_t}$, and that $\mathrm A_1\oplus \mathrm A_1$ does not have any proper \textit{even} integral overlattice. However, $Y_1\cong Y_2$ is a different K3 surface, and it must have larger discriminant, so the restriction map cannot be surjective.
\qed

\begin{remark}\normalfont\label{remark_only2}
   The two counterexamples to ICT$_{\Z}$ described in Section \ref{sec_LICTZ} and Theorem \ref{thm_2nd-counterexample} are the only ones possible using the given techniques. This is because there are exactly two extremal elliptic K3 surfaces (see \cite{Shimada-Zhang01}) that 
\begin{itemize}
\item[(i)] have three star fibers, and
\item[(ii)] are isomorphic to the Kummer surface of a product of elliptic curves.
\end{itemize}
Condition (i) is necessary for the defective quadratic base change with one free parameter, and Condition (ii) guarantees the existence of the Shioda--Inose structure (cf. Lemma \ref{lemma_ShiodaInose}), which computes the transcendental lattice of one of the limiting K3 surfaces which has small discriminant.

\end{remark}

\appendix
\section{SNC varieties}\label{sec_SNC}
In this section, we collect some topological results of SNC varieties used in this paper. 

\subsection{SNC varieties and hyperplane sections}

\begin{lemma}\label{lemma_LefHyp-coh}
    Let $X$  be a projective reducible variety of SNC type, and pure dimension $n$. Let $Y\subset X$ be a general hyperplane section.  Then 
\begin{enumerate}
 \item $H_i(X,Y; \Z)=0$ for $i < n$.
 \item $H^i(X,Y; \Z)=0$ for $i < n$.  
 \item $H^{n}(X,Y;\Z)$ is torsion-free.
\end{enumerate}     
\end{lemma}

\begin{proof}
    Since $X$ is lci, the first statement (1) follows from the fact \cite[Theorem 2.1.4]{HamLe85} or \cite[p.153, Theorem 1.2]{GorMac_88} that 
    the pair $(X,Y)$ is $(n-1)$-connected.  Alternatively, if $X_p$ and $Y_p$  denote  the disjoint union of intersections of $(p+1)$ components, we can deduce (1) from the usual Lefschetz theorem applied to the $E^1$ page of the
    Mayer--Vietoris spectral sequence
\[
E^1_{p,q}=H_q(X_p, Y_p;\mathbf Z)\Longrightarrow H_{p+q}(X,Y;\mathbf Z),
\]
 The next two statements (2) and (3) follow from (1) by the universal coefficient theorem.
\end{proof}




\subsection{Top cohomology of SNC variety}
\begin{lemma}\label{lemma_degMV}
   The top dimensional cohomology of an SNC variety is torsion-free. More precisely, let $Z=\cup_i Z_i$ be an SNC variety with pure dimension $n$, then $$H^{2n}(Z,\Z)\cong \oplus_i H^{2n}(Z_i,\Z)$$
   is generated by the fundamental classes of the irreducible components.
\end{lemma}
\begin{proof}
Let $Z_{i_0,\ldots,i_p}=Z_{i_0}\cap\cdots\cap Z_{i_p}$. There is a Mayer-Vietoris spectral sequence
$$E_1^{p,q}=\bigoplus_{i_0<\cdots<i_p} H^q(Z_{i_0,\ldots,i_p},\Z)\Rightarrow H^{p+q}(Z,\Z).$$
However, by dimension reason, $E_1^{p,q}=0$ unless $q\le 2(n-p)$. Hence the only nonzero term with $p+q=2n$ is 
$E_1^{0,2n}=\oplus_iH^{2n}(Z_i,\Z),$
and there is no nonzero differential out of $E_r^{0,2n}$ for $r\ge 1$. Thus 
$$H^{2n}(Z,\Z)\cong E_{\infty}^{0,2n}=E_1^{0,2n}= \oplus_iH^{2n}(Z_i,\Z).$$
\end{proof}

\section{Correspondence with the associated K3 surface}\label{sec_periodmap}
In this section, we study the lattice invariants of the elliptic-elliptic surfaces with constant VHS from Sections \ref{sec_LICTZ} and \ref{sec_addcounterexample}. Here, we denote them as $S_t^3$ and $S_t^4$, with discriminants 3 and 4, respectively. By Lemma \ref{lemma:sp1} and Theorem \ref{thm_2nd-counterexample}, their transcendental lattices are isomorphic to those of the Vinberg's most algebraic K3 surfaces $Y^3$ and $Y^4$ \cite{Vin83} with discriminants 3 and 4, respectively:
$$T_{S_t^3}\cong T_{Y^3}\cong \mathrm A_2,\ T_{S_t^4}\cong T_{Y^4}\cong \mathrm A_1\oplus \mathrm A_1.$$

We have the following result analogous to \cite[p.129 Corollary]{ShiIno77}.
\begin{proposition}\label{prop_alg-corr}
    The isometry between the transcendental Hodge structures on the elliptic-elliptic surface $S_t^3$ (resp. $S_t^4$) and its associated K3 surface $Y^3$ (resp. $Y^4$) 
    $$\Phi:T_{S_t^i}\xrightarrow{\cong} T_{Y^i},\ i=3,4$$
    is induced by an algebraic correspondence.
\end{proposition}
\begin{proof}
There is a diagram of rational double covers, which arise from defective quadratic base changes of the Kummer K3 surface $X^i/\PP^1$ from \eqref{eqn_Kum(E0*E0)} and \eqref{eqn_Kum(E_i*E_i)}:
\begin{figure}[ht]
    \centering
\begin{equation*}
\begin{tikzcd}
S_t^i/E_t\arrow[dr,dashed,"\pi_1"] && Y^i/\mathbb P^1\arrow[dl,dashed,"\pi_2"']\\
&X^i/\mathbb P^1,
\end{tikzcd}
\end{equation*}
\end{figure}

Here, $\pi_2$ is a Shioda--Inose double cover (cf. Lemma \ref{lemma_ShiodaInose}). Lemma \ref{lemma:sp1} and Theorem \ref{thm_2nd-counterexample} show that $(\pi_1)_*:T_{S_t^i}\to T_{X^i}$ is surjective, and by Remark \ref{remark_Morrison}, it scales the intersection pairing by $2$. In particular, the diagram above induces Hodge isometries 
  $$T_{S_t^i}(4)\overset{(\pi_1)_*}{\cong} T_{X^i}(2)\overset{(\pi_2)^*}{\cong} 2T_{Y^i}.$$ Therefore, the Hodge isometry $\Phi$ is equal to $\frac{1}{2}\pi_2^*(\pi_1)_*$, which is induced by $\frac{1}{2}$ the closure of the fibered product of $\pi_1$ and $\pi_2$ inside $S_t^i\times Y^i$.
\end{proof}
In \cite[Theorem~6.3]{Morrison84}, Morrison characterized which K3 surfaces admit a Shioda--Inose structure. Motivated by Proposition \ref{prop_alg-corr}, we conclude with an analogous question for elliptic-elliptic surfaces:


\begin{question}
 Let $\pi\colon S \dashrightarrow X$ be a rational double cover as in Corollary~\ref{cor_ell-ell-by-base-change} (2), where $X$ is an elliptic K3 surface with three star fibers and $S$ is the resulting elliptic--elliptic surface. 
For which such pairs $(S,X)$ does $\pi_*$ induce a Hodge isometry
\[
T_S(2)\ \cong\ T_X ?
\]
\end{question}



\section{Invariant divisor classes}\label{sec_divisor}
Since our counterexamples to the LICT$_{\Z}$ are transcendental classes, one may ask whether divisor classes satisfy the ICT$_{\Z}$ locally or globally. When $q(X_t)=0$, then the answer is yes by Theorem \ref{thm_H^2}. Here we drop this assumption and give another sufficient criterion for surfaces.

\begin{lemma}\label{lemma_ICTZ-div}
    Let $\mathcal X\to B$ be a family of surfaces, where $B$ can be an analytic disk or a smooth projective variety. 
    Suppose for a Zariski open subspace $U\subset B$ and every $t\in U$, $\alpha_t=[Z_t]$ is represented by an irreducible curve $Z_t\subset X_t$ with negative self-intersection, and moreover $\alpha$ is invariant under the monodromy action of $\pi_1(U,t)$. 
 Then there is a divisor $Z\subset\mathcal X$ such that $Z|_{X_t}=Z_t$. In particular, the \textup{ICT}$_{\Z}$ holds for $[Z_t]$. 
\end{lemma}
\begin{proof}
   The assumption $Z_t^2<0$ implies that $H^0(N_{Z_t|X_t})=0$ and the Hilbert scheme $\textup{Hilb}(X_t)$ is reduced and is 0-dimensional at $Z_t$. For each point $t\in U$, there is an analytic neighborhood $V_t$ of $t$ such that $\sigma_t:t\mapsto Z_t$ defines a section of the relative Hilbert scheme $\textup{Hilb}(\mathcal X_{V_t}/V_t)$. The assumption that the class $\alpha_t=[Z_t]$ is monodromy-invariant implies that $\{\sigma_t\}$ patches to a global section $\sigma$ on $\textup{Hilb}(\mathcal X_{U}/U)$. The universal correspondence defines a divisor $Z_U\subset\mathcal X_U$ such that $Z_U|_{X_t}=Z_t$. Finally, we take the closure to obtain a divisor $Z$ on $\mathcal X$.
\end{proof}

\begin{corollary} \label{cor_ICTZ-divisor}
   The \textup{ICT}$_{\Z}$ for divisor classes holds for the semistable families of elliptic surfaces in Theorems \ref{main_thm} and \ref{intro_thm-failureGICTZ}. 
\end{corollary}
\begin{proof}
  For these semistable families, the global monodromy is trivial, so every class in $NS(S_t)$ is invariant. To find generators with negative self-intersection, recall that $S_t$ is an elliptic-elliptic surface which has singular fibers of types $IV^*,IV$. The irreducible components of those fibers have self-intersection $-2$, and any section curves have self-intersection $-1$. These generate $NS(S_t)$ by the Shioda-Tate formula, so the claim follows from Lemma \ref{lemma_ICTZ-div}.
\end{proof}

\bibliographystyle{alpha}
\bibliography{LICTZ}

\end{document}